\documentclass[a4paper]{amsart}
\usepackage[all]{xy}

\usepackage{pstricks-add}
\usepackage{float}

\usepackage[colorlinks=true]{hyperref}
\usepackage{enumerate}
\usepackage{amssymb}
\usepackage{amsbsy}
\usepackage{comment}
\usepackage{aurical}

\newtheorem{thm}{Theorem}[section]

\newtheorem{prop}[thm]{Proposition}
\newtheorem{lem}[thm]{Lemma}
\newtheorem{cor}[thm]{Corollary}

\theoremstyle{definition}
\newtheorem{defn}[thm]{Definition}

\theoremstyle{remark}
\newtheorem{rem}[thm]{Remark}

\renewcommand{\d}{\mathrm d}                           

\renewcommand{\d}{\mathrm d}               

\def\bb#1#2{\left\{#1,#2\right\}}
\def\Dorf#1#2{\big[\!\!\big[#1,#2\big]\!\!\big]}

\begin{document}
\title{Hypersymplectic structures with torsion on Lie algebroids}

\author{P. Antunes}
\address{CMUC, Department of Mathematics, University of Coimbra, 3001-454 Coimbra, Portugal}
\email{pantunes@mat.uc.pt}
\author{J.M. Nunes da Costa}
\address{CMUC, Department of Mathematics, University of Coimbra, 3001-454 Coimbra, Portugal}
\email{jmcosta@mat.uc.pt}

\begin{abstract}
Hypersymplectic structures with torsion on Lie algebroids are investigated. We show that each hypersymplectic structure with torsion on a Lie algebroid determines three Nijenhuis morphisms. From a contravariant point of view, these structures are twisted Poisson structures. We prove the existence of a one-to-one correspondence between hypersymplectic structures with torsion and hyperk\"{a}hler structures with torsion.  We show that given a Lie algebroid with a hypersymplectic structure with torsion, the deformation of the Lie algebroid structure by any of the transition morphisms does not affect the hypersymplectic structure with torsion. We also show that if a triplet of $2$-forms is a hypersymplectic structure with torsion on a Lie algebroid $A$, then the triplet of the inverse bivectors is a hypersymplectic structure with torsion for a certain Lie algebroid structure on the dual $A^*$, and conversely. Examples of hypersymplectic structures with torsion are included.
\end{abstract}

\maketitle

%
\section*{Introduction}             %
\label{sec:introduction}           %
Hypersymplectic structures with torsion on Lie algebroids were introduced in \cite{AC14a}, in relation with hypersymplectic structures on Courant algebroids, when these Courant algebroids are doubles of quasi-Lie and proto-Lie bialgebroids. In fact, while looking for examples of hypersymplectic structures on Courant algebroids we found in \cite{AC14a}, in a natural way, hypersymplectic structures with torsion on Lie algebroids. The aim of this article is to study these structures on Lie algebroids. We look at them as the ``symplectic version'' of hyperk\"{a}hler structures with torsion and also from a contravariant perspective, as twisted Poisson structures.

 A triplet $(\omega_1, \omega_2, \omega_3)$ of non-degenerate $2$-forms on a Lie algebroid $(A, \mu)$ is a hypersymplectic structure with torsion if the transition morphisms $N_i, i=1,2,3,$ satisfy $N_i^2=-{\rm id}_A$ and $N_1 \d \omega_1= N_2 \d \omega_2=N_3 \d \omega_3$ (Definition~\ref{def_HST}). These structures can be viewed as a generalization of hypersymplectic structures on Lie algebroids, a structure we have studied in \cite{AC13}. We prove that they are in a one-to-one correspondence with  hyperk\"{a}hler structures with torsion~\cite{GP00,CS08}, also known as HKT structures, which first appeared in \cite{HP96} in relation with sigma models in string theory.
Let us briefly recall what a HKT manifold is. Let $M$ be  a hyperhermitian manifold, i.e., a manifold equipped with three complex structures $N_1$, $N_2$ and $N_3$ satisfying $N_1N_2=-N_2N_1=N_3$ and a metric $g$ compatible with the three complex structures. If  there exists a linear connection $\nabla$ on $M$ such that $\nabla g=0$, $\nabla N_1= \nabla N_2=\nabla N_3=0$ and $H$, defined by $H(X,Y,Z)= g(X,\nabla_Y Z-\nabla_Z Y - [Y,Z])$, is a $3$-form on $M$, then $M$ is a  HKT manifold.
In \cite{GP00} it was proved that the condition that $H$ be a $3$-form can be replaced by the following equivalent requirement: $N_1 \d \omega_1= N_2 \d \omega_2=N_3 \d \omega_3$, where $\omega_1$, $\omega_2$ and $\omega_3$ are the associated K\"{a}hler forms.
Later, in \cite{CS08}, the authors showed that the assumption that $N_1$, $N_2$ and $N_3$ are Nijenhuis can be removed from the definition of HKT manifold, because the equalities $N_1 \d \omega_1= N_2 \d \omega_2=N_3 \d \omega_3$ imply the vanishing of the Nijenhuis torsion of the morphisms $N_1$, $N_2$ and $N_3$.
Thus, the definition of a HKT manifold can be simplified, requiring that it is an almost hyperhermitian manifold satisfying $N_1 \d \omega_1= N_2 \d \omega_2=N_3 \d \omega_3$. Here, we extend the notion of HKT structure to Lie algebroids. Our definition of hypersymplectic structure with torsion and HKT structure is more general than the usual one, since we consider the cases of (almost) complex and para-complex morphisms $N_i$.
We also look at hypersymplectic structures with torsion on Lie algebroids from a different perspective,  by presenting an alternative definition in terms of bivectors instead of $2$-forms (Theorem \ref{thm_HST_equivalent_contravariant_form}). These bivectors are twisted Poisson, also known as Poisson with a $3$-form background \cite{sw01}. Moreover, we prove that the almost complex morphisms $N_1$, $N_2$ and $N_3$, that are constructed out of the $2$-forms and the twisted Poisson bivectors, are in fact Nijenhuis morphisms (Theorem \ref{Thm_PHST->Ni_Nijenhuis}). In other words, a hypersymplectic structure with torsion on a Lie algebroid $A$ determines three Nijenhuis morphisms and three twisted Poisson bivectors. It is well known~\cite{magriYKS} that if $(A,\mu)$ is a Lie algebroid and $N$ is a Nijenhuis morphism, the deformation (in a certain sense) of $\mu$  by $N$ yields a new Lie algebroid structure on A, that we denote by $\mu_N$.  On the other hand, if a Lie algebroid $(A,\mu)$ is equipped with a nondegenerate twisted Poisson bivector $\pi$, the dual vector bundle $A^*$ inherits a Lie algebroid structure given by $\mu_{\pi}+\frac{1}{2}\bb{\omega}{[\pi, \pi]}$, where $\omega$ is the inverse of $\pi$ (Proposition \ref{prop_Twist_Poisson_induces_Lie_algebroid_in_dual}).
So, two natural questions arise.

\begin{enumerate}
\item [1.]
\emph{If $(\omega_1, \omega_2, \omega_3)$ is a hypersymplectic structure with torsion on a Lie algebroid $(A, \mu)$, does it remain a hypersymplectic structure with torsion on the Lie algebroid $(A,\mu_{N_i})$, for $i=1,2,3$?}

\

In Theorem \ref{thm_PHST_on_A_iff_PHST_on_A_N} we answer this question positively, showing that $(\omega_1, \omega_2, \omega_3)$ is a hypersymplectic structure with torsion on $(A, \mu)$ if and only if it is a hypersymplectic structure with torsion on $(A,\mu_{N_i})$, for $i=1,2,3$.

\

\

\item [2.]
\emph{ Does a hypersymplectic structure with torsion on a Lie algebroid $A$ induce a hypersymplectic structure with torsion on the Lie algebroid $A^*$?}

 \

 The  answer is given in Theorem \ref{Thm_PHST_on_A_iff_PHST_on_A*}, where we prove that $(\omega_1, \omega_2, \omega_3)$ is a hypersymplectic structure with torsion on $(A, \mu)$ if and only if $(\pi_1, \pi_2, \pi_3)$ is a hypersymplectic structure with torsion on $\left(A^*, -\mu_{\pi_i}-\frac{1}{2}\bb{\omega_i}{[\pi_i, \pi_i]}\right)$, $i=1,2,3$, where $\pi_i$ is the inverse of $\omega_i$.

\end{enumerate}

Question 2 was answered in \cite{AC14} for the case of hypersymplectic structures (without torsion) on Lie algebroids.

\


In Section~\ref{sec1} we define  {\large${\boldsymbol\varepsilon}$}-hypersymplectic structures
with torsion on a Lie algebroid and present some of their properties. In Section~\ref{sec2} we characterize them in terms of twisted Poisson bivectors. The structure induced on the base manifold of a Lie algebroid with a hypersymplectic structure with torsion is described in Section~\ref{sec3}. In Section~\ref{sec4} we introduce the notion of hyperk\"{a}hler structure with torsion on a Lie algebroid and prove the existence of a one-to-one correspondence between hypersymplectic structures with torsion and hyperk\"{a}hler structures with torsion on a Lie algebroid (Theorem \ref{Thm_1-1_correspondence}). Section~\ref{sec5} contains three examples of hypersymplectic structures with torsion on $\mathbb R^8$, on $\mathfrak{su}(3)$ and on the tangent bundle of $S^3 \times (S^1)^5$, respectively. In Section~\ref{sec6} we start by recalling the definition of (para-)hypersymplectic structures on a pre-Courant algebroid.
Then, we concentrate on pre-Courant structures on the vector bundle $A \oplus A^*$, to study the Nijenhuis torsion of an endomorphism of $A \oplus A^*$ of type ${\mathcal D}_N= N \oplus (-N^*)$, with $N: A \to A$. Namely, we establish relations between the Nijenhuis torsion of ${\mathcal D}_N$ and the Nijenhuis torsion of $N$ (Propositions~\ref{prop_relation_T(JN)_and_T(N)} and \ref{torsion_TN_torsion_N}). In Proposition~\ref{prop_C_mu(Ti,Tj)=-2[I,J]_FN} we show how the Fr\"{o}licher-Nijenhuis bracket on a Lie algebroid can be seen in the pre-Courant algebroid setting.
Section~\ref{sec7} contains the main results of the paper. Given a (para-)hypersymplectic structure with torsion on a Lie algebroid, we prove that the transition morphisms are Nijenhuis (Theorem~\ref{Thm_PHST->Ni_Nijenhuis}) and pairwise compatible with respect to the Fr\"{o}licher-Nijenhuis bracket (Proposition~\ref{prop_PHST->[Ni,Ni]=0}), and that the twisted Poisson bivectors are compatible with respect to the Schouten-Nijenhuis bracket of the Lie algebroid (Proposition~\ref{prop_PHST->[pi_i,pi_j]=0}). The results of \cite{AC14a} are extensively used in the proofs of Section~\ref{sec7}.

Most of the computations along the paper are done using the big bracket~\cite{YKS92}. 
For further details on
Lie and pre-Courant algebroids in the supergeometric setting, see \cite{royContemp,roy,voronov,vaintrob}.

\

\noindent \emph{Notation}: We consider $1, 2$ and $3$ as the representative elements of the equivalence classes of $\mathbb{Z}_3$, i.e., $\mathbb{Z}_3:=\{[1],[2],[3]\}$. Along the paper, although we omit the brackets, and write $i$ instead of $[i]$, the indices are elements in $\mathbb{Z}_3:=\mathbb{Z}/{3\mathbb{Z}}$.

%
\section{Hypersymplectic structures with torsion}         
\label{sec1}                                              
%

Let $(A,\mu)$ be a  Lie algebroid and take three non-degenerate $2$-forms $\omega_1, \omega_2$ and $\omega_3 \in \Gamma(\wedge^2 A^*)$ with inverse  $\pi_1, \pi_2$ and $\pi_3 \in \Gamma(\wedge^2 A)$, respectively. We define the \textit{transition morphisms} $N_1,N_2$ and $N_3: A \to A$, by setting
\begin{equation}\label{def_Ni}
  N_i:=\pi_{i-1}^{\sharp}\circ\omega_{i+1}^{\flat}, \,\,\, i \in \mathbb{Z}_3.
\end{equation}
In~(\ref{def_Ni}), we consider the usual vector bundle maps $\pi^\# : A^* \to A$ and $\omega^\flat:A \to A^*$, associated to a bivector $\pi \in \Gamma(\bigwedge^2 A)$ and a $2$-form $\omega  \in \Gamma(\bigwedge^2 A^*)$, respectively, which are given by
$ \langle \beta, \pi^\# ( \alpha) \rangle= \pi(\alpha, \beta)$ and  $\langle \omega^\flat ( X), Y \rangle= \omega(X,Y)$,
for all  $\alpha, \beta \in \Gamma(A^*)$ and  $X, Y \in \Gamma(A)$.

In what follows we shall consider the parameters $\varepsilon_i=\pm 1, i=1,2,3,$ and the triplet {\large${\boldsymbol\varepsilon}$}$=(\varepsilon_1,\varepsilon_2,\varepsilon_3)$.

\begin{defn}\label{def_HST}
    A triplet $(\omega_1, \omega_2, \omega_3)$ of non-degenerate $2$-forms on a Lie algebroid $(A, \mu)$ is an \textit{{\large${\boldsymbol\varepsilon}$}-hypersymplectic structure with torsion} if
    \begin{equation}\label{eq_Ij_almostCPS}
        {N_i}^2=\varepsilon_i {\rm id}_{A}, \quad i=1,2,3, \quad \rm{and} \quad  \varepsilon_2 N_1 \d \omega_1=\varepsilon_3 N_2 \d \omega_2=\varepsilon_1 N_3 \d \omega_3,
    \end{equation}
    where $N_i$, $i=1,2,3$, is given by (\ref{def_Ni}) and $$N_i\d \omega_i(X,Y,Z):= \d \omega_i(N_i X, N_i Y, N_i Z),\textrm{ for all }X,Y,Z \in \Gamma(A).$$
\end{defn}

When the non-degenerate $2$-forms $\omega_1, \omega_2$ and $\omega_3$ are closed,  so that they are symplectic forms and $\varepsilon_2 N_1 \d \omega_1=\varepsilon_3 N_2 \d \omega_2=\varepsilon_1 N_3 \d \omega_3$ is trivially satisfied, the triplet $(\omega_1, \omega_2, \omega_3)$ is an {\large${\boldsymbol\varepsilon}$}-hypersymplectic structure  on $(A, \mu)$ \cite{AC13}.


Having an {\large $\boldsymbol\varepsilon$}-hypersymplectic structure with torsion $(\omega_1, \omega_2, \omega_3)$ on a
Lie algebroid $A$ over $M$,
we define a map
\begin{equation} \label{def_g}
g : A \times A \longrightarrow \mathbb{R},\quad \quad g(X,Y)=\langle g^\flat (X), Y\rangle,
\end{equation}
 where $g^\flat: A\longrightarrow A^*$ is the vector bundle morphism defined by
    \begin{equation}\label{first_defn _g}
    g^\flat:=\varepsilon_3\varepsilon_2\ {\omega_3}^{\flat} \circ {\pi_1}^{\sharp} \circ {\omega_2}^{\flat}.
    \end{equation}
The definition of $g^\flat$ is not affected by a circular permutation of the indices in~(\ref{first_defn _g}), that is,
\begin{equation} \label{second_defn_g}
g^\flat=\varepsilon_{i-1}\varepsilon_{i+1}\ {\omega_{i-1}}^{\flat} \circ {\pi_i}^{\sharp} \circ {\omega_{i+1}}^{\flat}, \,\, \, i \in \mathbb{Z}_3.
\end{equation}
Moreover,
\begin{equation}  \label{skew_or_not}
(g^\flat)^* = - \varepsilon_1\varepsilon_2\varepsilon_3\ g^\flat,
\end{equation}
which means that
 $g$ is symmetric or skew-symmetric, depending on the sign of $\varepsilon_1\varepsilon_2\varepsilon_3$. 
An important property of $g$ is the following one:
\begin{equation}
g(N_i X, N_i Y) = \varepsilon_{i-1}\varepsilon_{i+1}\ g(X,Y), \quad X, Y \in \Gamma(A).
\end{equation}
Notice that $g^\flat$ is invertible and, using its inverse, we may define a map \mbox{$g^{-1}: A^* \times A^* \longrightarrow \mathbb{R}$,} by setting
\begin{equation}  \label{g_inverse}
g^{-1}(\alpha,\beta):=\langle \beta, (g^\flat)^{-1}(\alpha) \rangle,
\end{equation}
for all $\alpha,\beta \in \Gamma(A^*)$.



\

All the algebraic properties of {\large${\boldsymbol\varepsilon}$}-hypersymplectic structures on Lie algebroids, proved in Propositions 3.7, 3.9 and 3.10 in \cite{AC13}, hold in the case of {\large${\boldsymbol\varepsilon}$}-hypersymplectic structures with torsion.
\section{The contravariant approach}        
\label{sec2}                                

In this section we give a contravariant characterization of an {$\Large\varepsilon$}-hypersymplectic structure with torsion on a Lie algebroid.

We shall need the next lemma, that can be easily proved.

\begin{lem}  \label{pi_omega_inverses}
Let $\omega$ be a non-degenerate $2$-form on a Lie algebroid $(A, \mu)$, with inverse $\pi$. Then,
\begin{equation} \label{pi_inverse omega}
\frac{1}{2}[\pi, \pi]= \d \omega \left(\pi^{\sharp}(\cdot), \pi^{\sharp}(\cdot) , \pi^{\sharp}(\cdot)\right)
\end{equation}
 or, equivalently,
\begin{equation} \label{omega_inverse pi}
\d \omega = \frac{1}{2} [\pi, \pi] \left(\omega^{\flat}(\cdot), \omega^{\flat}(\cdot), \omega^{\flat}(\cdot)\right),
\end{equation}
where $[\cdot,\cdot]$ stands for the Schouten-Nijenhuis bracket of multivectors on $(A, \mu)$.
\end{lem}

Equation (\ref{pi_inverse omega}) means that $\pi$ is a \emph{twisted Poisson} bivector on $(A, \mu)$, also known as Poisson bivector with the $3$-form background ${\rm d}\omega$ \cite{sw01}.

Let $(\omega_1, \omega_2, \omega_3)$ be an {$\Large\varepsilon$}-hypersymplectic structure with torsion on $(A, \mu)$ and let us denote by $H$ the $3$-form
$$H:=\varepsilon_2 N_1 \d \omega_1=\varepsilon_3 N_2 \d \omega_2= \varepsilon_1 N_3 \d \omega_3.$$
So, for all $X,Y,Z \in \Gamma(A)$,
\begin{equation} \label{def_H}
H(X,Y,Z)= \varepsilon_{i+1} \d \omega_i(N_i X, N_i Y, N_i Z), \,\,\, i=1,2,3,
\end{equation}
or, equivalently,
\begin{equation} \label{2nd_def_H}
\d \omega_i(X,Y,Z)=\varepsilon_{i}\varepsilon_{i+1} H(N_i X, N_i Y, N_i Z), \,\,\, i=1,2,3.
\end{equation}

Now we prove the main result of this section, that gives a characterization of {$\Large\varepsilon$}-hypersymplectic structures with torsion on a Lie algebroid.

\begin{thm}\label{thm_HST_equivalent_contravariant_form}
A triplet $(\omega_1, \omega_2, \omega_3)$ of non-degenerate $2$-forms on a Lie algebroid  $(A, \mu)$ with inverses $\pi_1, \pi_2$ and $\pi_3$, respectively, is an {$\Large\varepsilon$}-hypersymplectic structure with torsion on $A$ if and only if
\begin{equation*}
        {N_i}^2=\varepsilon_i {\rm id}_{A}, \quad i=1,2,3, \quad \rm{and} \quad \varepsilon_1[\pi_1, \pi_1]=  \varepsilon_2 [\pi_2, \pi_2]= \varepsilon_3 [\pi_3, \pi_3],
        \end{equation*}
        with $N_i$ given by (\ref{def_Ni}).
\end{thm}
\begin{proof}

Assume that $\varepsilon_2 N_1 \d \omega_1=\varepsilon_3 N_2 \d \omega_2=\varepsilon_1 N_3 \d \omega_3$. 
Using (\ref{pi_inverse omega}), (\ref{2nd_def_H}) and formula $N_i \circ \pi_i^{\sharp}= \varepsilon_{i-1}\varepsilon_{i} (g^{-1})^{\sharp}$ (see~\cite{AC13}), we have
\begin{eqnarray*}
\varepsilon_i[\pi_i,\pi_i]&=& 2 \varepsilon_i \d \omega_i \left(\pi_i^{\sharp}(\cdot), \pi_i^{\sharp}(\cdot), \pi_i^{\sharp}(\cdot)\right) \\
&=& 2 \varepsilon_1 \varepsilon_2 \varepsilon_3 H\left( (g^{-1})^{\sharp}(\cdot), (g^{-1})^{\sharp}(\cdot), (g^{-1})^{\sharp}(\cdot)\right),
\end{eqnarray*}
for $i \in \mathbb{Z}_3$; thus,
$$ \varepsilon_1 [\pi_1, \pi_1]= \varepsilon_2 [\pi_2, \pi_2]= \varepsilon_3 [\pi_3, \pi_3].$$

Conversely, assume that $ \varepsilon_1 [\pi_1, \pi_1]=  \varepsilon_2 [\pi_2, \pi_2]= \varepsilon_3 [\pi_3, \pi_3]$ and let us set $\psi:=\frac{1}{2}\varepsilon_i [\pi_i, \pi_i]$. Then, using (\ref{omega_inverse pi}) and formula $\omega_i^\flat\circ N_i= \varepsilon_{i-1} g^{\flat}$ (see~\cite{AC13}), we get
\begin{eqnarray*}
\varepsilon_{i+1}\d \omega_i \left(N_i(\cdot), N_i(\cdot), N_i(\cdot)\right)&=& \frac{1}{2}\varepsilon_{i+1} [\pi_i,\pi_i] \left(\omega_i ^{\flat} \circ N_i(\cdot), \omega_i ^{\flat} \circ N_i(\cdot), \omega_i ^{\flat} \circ N_i(\cdot)\right)\\
&=& \varepsilon_1 \varepsilon_2 \varepsilon_3 \psi \left(g^{\flat}(\cdot), g^{\flat}(\cdot), g^{\flat}(\cdot)\right).
\end{eqnarray*}
Therefore, we conclude
$$\varepsilon_2 N_1 \d \omega_1=\varepsilon_3 N_2 \d \omega_2=\varepsilon_1 N_3 \d \omega_3.$$
\end{proof}

\begin{rem}
When there is an {$\Large\varepsilon$}-hypersymplectic structure with torsion on a Lie algebroid $(A, \mu)$, this Lie algebroid is equipped with a triplet $(\pi_1, \pi_2, \pi_3)$ of twisted Poisson bivectors that share, eventually up to a sign, the same obstruction to be Poisson. This obstruction is denoted by $\psi$ in the proof of Theorem \ref{thm_HST_equivalent_contravariant_form}.
\end{rem}

%
\section{Structures induced on the base manifold}
\label{sec3}
%

It is well known that a symplectic structure $\omega$ on a Lie algebroid $A \to M$ induces a Poisson structure on the base manifold $M$. The Poisson bivector $\pi_{ _{M}}$ on $M$ is defined by
\begin{equation} \label{pi_M}
\pi_{ _{M}}^\sharp= \rho \circ \pi^\sharp \circ \rho^*,
 \end{equation}
 where $\rho$ is the anchor map and $\pi \in \Gamma(\wedge^2 A)$ is the Poisson bivector on $A$ which is the inverse of $\omega$.  In the case where $\omega$ is non-degenerate but not necessarily closed, so that $\pi$ defines a Poisson structure with the $3$-form background $d \omega$ on the Lie algebroid $A$ (see (\ref{pi_inverse omega})), the base manifold $M$ has an induced Poisson structure with a $3$-form background, provided that $d \omega \in \Gamma(\wedge^3 A^*)$ is the pull-back by $\rho$ of a closed $3$-form $\phi_{ _{M}} \in \Omega^3(M)$. In fact, if $d \omega = \rho^*(\phi_{ _{M}})$ then, a straightforward computation gives
 \begin{equation*}
[\pi, \pi]= 2 d \omega  \left(\pi^{\sharp}(\cdot), \pi^{\sharp}(\cdot), \pi^{\sharp}(\cdot)\right)
\end{equation*}
which implies that
\begin{equation*}
[\pi_{ _{M}}, \pi_{ _{M}}]_{ _{M}}= 2 \phi_{ _{M}} \left(\pi_{ _{M}}^{\sharp}(\cdot), \pi_{ _{M}}^{\sharp}(\cdot), \pi_{ _{M}}^{\sharp}(\cdot)\right),
\end{equation*}
where $\pi_{ _{M}}$ is given by (\ref{pi_M}) and $[\cdot,\cdot]_{ _{M}}$ stands for the Schouten-Nijenhuis bracket of multivectors on $M$.

The next proposition is now immediate.

\begin{prop} \label{twisted_on_M}
Let $(\omega_1, \omega_2, \omega_3)$ be an {$\Large\varepsilon$}-hypersymplectic structure with torsion on a Lie algebroid  $(A, \mu)$ over $M$ with inverses $\pi_1, \pi_2$ and $\pi_3$, respectively. Let $\rho$ be the anchor map of $A$. If $d \omega_i= \rho^*(\phi_{ _{M}}^i)$, $i=1,2,3$, with $\phi_{ _{M}}^i$ a closed $3$-form on the manifold $M$, then $M$ is equipped with three twisted Poisson bivectors $\pi_{ _{M}}^1, \pi_{ _{M}}^2$ and $\pi_{ _{M}}^3$, defined by (\ref{pi_M}); the $3$-forms background are $ \phi_{ _{M}}^1$, $\phi_{ _{M}}^2$ and  $ \phi_{ _{M}}^3$, respectively. Moreover, the three bivectors on $M$ share, eventually up to a sign, the same obstruction to be Poisson.
\end{prop}

We should stress that if $\omega_i=\rho^*(\omega_{ _{M}}^i)$ for some $2$-forms $\omega_{ _{M}}^i$ on $M$, $i=1,2,3$, then $d \omega_i = \rho^*(\delta \omega_{ _{M}}^i)$ with $\delta$ the De Rham differential on $M$. So,  the assumptions of
 Proposition~\ref{twisted_on_M} are satisfied and $$[\pi_{ _{M}}^i, \pi_{ _{M}}^i]_{ _{M}}= 2 \delta \omega_{ _{M}}^i  \left((\pi_{ _{M}}^i)^{\sharp}(\cdot), (\pi_{ _{M}}^i)^{\sharp}(\cdot), (\pi_{ _{M}}^i)^{\sharp}(\cdot)\right),$$ $i=1,2,3$. However, although the bivectors $\pi_{ _{M}}^1, \pi_{ _{M}}^2$ and $\pi_{ _{M}}^3$ on $M$ satisfy $\varepsilon_1[\pi_{ _{M}}^1, \pi_{ _{M}}^1]_{ _{M}}=  \varepsilon_2 [\pi_{ _{M}}^2, \pi_{ _{M}}^2]_{ _{M}}= \varepsilon_3 [\pi_{ _{M}}^3, \pi_{ _{M}}^3]_{ _{M}}$,  in general, the triplet $(\omega_{ _{M}}^1,\omega_{ _{M}}^2,\omega_{ _{M}}^3)$ does not define an {\large$\boldsymbol{\varepsilon}$}-hypersymplectic structure with torsion on $M$, because $\omega_{ _{M}}^i$ is not the inverse of $\pi_{ _{M}}^i$.

%
\section{Hypersymplectic structures with torsion versus hyperk\"{a}hler structures with torsion}
\label{sec4}
%

In this section we consider {\large$\boldsymbol{\varepsilon}$}-hypersymplectic structures with torsion such that $\varepsilon_1\varepsilon_2\varepsilon_3=-1$ and we prove that these structures are in one-to-one correspondence with (para-)hyperk\"{a}hler structures with torsion, a notion we shall define later. First, let us consider two different cases of an {\large$\boldsymbol{\varepsilon}$}-hypersymplectic structure with torsion such that $\varepsilon_1\varepsilon_2\varepsilon_3=-1$.

\begin{defn}\cite{AC14a}
    Let $(\omega_1, \omega_2, \omega_3)$ be an {\large$\boldsymbol{\varepsilon}$}-hypersymplectic structure with torsion on a Lie algebroid $(A, \mu)$, such that $\varepsilon_1\varepsilon_2\varepsilon_3=-1$.
    \begin{itemize}
        \item If $\varepsilon_1=\varepsilon_2=\varepsilon_3=-1$, then $(\omega_1, \omega_2, \omega_3)$ is said to be a \emph{hypersymplectic} structure with torsion on $A$.
        \item Otherwise, $(\omega_1, \omega_2, \omega_3)$ is said to be a \emph{para-hypersymplectic} structure with torsion on $A$.
    \end{itemize}
\end{defn}


%

For a (para-)hypersymplectic structure with torsion on a Lie algebroid $(A, \mu)$, the morphism $g^\flat$ defined by (\ref{first_defn _g}) determines a pseudo-metric on $(A, \mu)$ \cite{AC13}.

\

Hyperk\"{a}hler structures with torsion on manifolds were introduced in \cite{HP96} and studied in \cite{GP00} and in \cite{CS08}. The definition extends to the Lie algebroid setting in a natural way.
Let $(A, \mu)$ be a Lie algebroid and consider a map $\mathtt{g} : A \times A \to \mathbb{R}$ and endomorphisms $I_1, I_2, I_3:A\to A$ such that, for all $i\in \mathbb{Z}_3$ and $X,Y \in \Gamma(A)$,
\begin{align}
    i)&\ \mathtt{g}\text{ is a pseudo-metric};\nonumber \\
    ii)&\ I_i^2=\varepsilon_i {\rm id}_A, \quad \text{ where }\varepsilon_i=\pm 1\text{ and }\varepsilon_1\varepsilon_2\varepsilon_3=-1;\nonumber\\
    iii)&\ I_3=\varepsilon_1\varepsilon_2 I_1 \circ I_2;\label{HKT_axioms}\\
    iv)&\ \mathtt{g}(I_i X, I_i Y) =  \varepsilon_{i-1}\varepsilon_{i+1}\mathtt{g}(X,Y);\nonumber\\
    v)&\ \varepsilon_2 I_1 \d \omega_1=\varepsilon_3 I_2 \d \omega_2=\varepsilon_1 I_3 \d \omega_3,\nonumber
\end{align}
where the $2$-forms $\omega_i, i=1,2,3$, which are called the \emph{K\"{a}hler forms}, are defined by
\begin{equation}\label{def_Kahler_forms}
\omega_i^\flat= \varepsilon_{i}\varepsilon_{i-1}\mathtt{g}^\flat \circ I_i.
\end{equation}

\begin{defn} \label{HKT_definition}
    A quadruple $(\mathtt{g}, I_1, I_2, I_3)$ satisfying (\ref{HKT_axioms})$i)-v)$, on a Lie algebroid $(A, \mu)$, is a
   \begin{itemize}
        \item \emph{hyperk\"{a}hler structure with torsion} on $A$ if $\varepsilon_1=\varepsilon_2=\varepsilon_3=-1$;
        \item \emph{para-hyperk\"{a}hler structure with torsion} on $A$, otherwise.
    \end{itemize}
\end{defn}

\begin{rem}
    Most authors consider that a (para-)hyperk\"{a}hler structure, with or without torsion, is equipped with a positive definite metric, while for us $\mathtt{g} : A \times A \to \mathbb{R}$ is a pseudo-metric, i.e., it is symmetric and non-degenerate.
\end{rem}

Note that, because $\varepsilon_1\varepsilon_2\varepsilon_3=-1$, on a (para-)hyperk\"{a}hler structure with torsion $(\mathtt{g}, I_1, I_2, I_3)$, we always have $I_i\circ I_j=-I_j\circ I_i$, for all $i, j\in \{1,2,3\},i\neq j$ (see Proposition 3.9 in~\cite{AC13}).

The next lemma establishes a relation between the pseudo-metric and the K\"{a}hler forms of a (para-)hyperk\"{a}hler structure with torsion on a Lie algebroid. Namely, it is shown that the pseudo-metric satisfies an equation similar to (\ref{second_defn_g}).

\begin{lem} \label{g_and Kahler}
Let $(\mathtt{g}, I_1, I_2, I_3)$ be a (para-)hyperk\"{a}hler structure with torsion on a Lie algebroid with associated K\"{a}hler forms $\omega_1$, $\omega_2$ and $\omega_3$. Then,
$$\mathtt{g}^\flat= \varepsilon_{i-1}\varepsilon_{i+1}\,{\omega_{i-1}}^{\flat} \circ {\pi_i}^{\sharp} \circ {\omega_{i+1}}^{\flat}, \,\, \, i \in \mathbb{Z}_3,$$
where $\pi_i$ is the inverse of $\omega_i$.
\end{lem}

\begin{proof}
It is enough to prove that $\mathtt{g}^\flat=\varepsilon_1\varepsilon_3 {\omega_{1}}^{\flat} \circ {\pi_2}^{\sharp} \circ {\omega_{3}}^{\flat}$.  From (\ref{def_Kahler_forms}) we have, on one hand,
\begin{equation} \label{index3}
\mathtt{g}^\flat= \varepsilon_2 \omega_3^{\flat}\circ I_3= \varepsilon_1 \omega_3^{\flat} \circ I_1 \circ I_2
\end{equation}
and, on the other hand,
\begin{equation} \label{index2}
\mathtt{g}^\flat= \varepsilon_1 \omega_2^{\flat} \circ I_2.
\end{equation}
From (\ref{index3}) and (\ref{index2}), we get
\begin{equation*}
\omega_3^{\flat} \circ I_1=\omega_2^{\flat}\, \Leftrightarrow \, I_1= \pi_3^{\sharp} \circ \omega_2^{\flat} \, \Leftrightarrow \, I_1= \varepsilon_1 \pi_2^{\sharp} \circ \omega_3^{\flat},
\end{equation*}
where we used $(I_1)^{-1}=\varepsilon_1 I_1$,
and so,  $\mathtt{g}^\flat= \varepsilon_3 \omega_1^{\flat}\circ I_1$ yields
$$\mathtt{g}^\flat= \varepsilon_1\varepsilon_3{\omega_{1}}^{\flat} \circ {\pi_2}^{\sharp} \circ {\omega_{3}}^{\flat}.$$
\end{proof}

At this point, we shall see that  (para-)hypersymplectic structures with torsion and (para-)hyperk\"{a}hler structures with torsion on a Lie algebroid are in one-to-one correspondence.

\begin{thm}\label{Thm_1-1_correspondence}
    Let $(A, \mu)$ be a Lie algebroid.
There exists a one-to-one correspondence between  (para-)hypersymplectic structures with torsion and (para-)hyperk\"{a}hler structures with torsion on $(A, \mu)$.
\end{thm}

\begin{proof}
Let $(\omega_1, \omega_2, \omega_3)$ be a (para-)hypersymplectic structure with torsion on $A$ and consider the endomorphisms $N_1,N_2, N_3$ and $g^\flat$ given by (\ref{def_Ni}) and (\ref{first_defn _g}), respectively. Since the equalities $N_1 \circ N_2 = \varepsilon_1 \varepsilon_2 N_3= - N_2 \circ N_1$ hold and $g$ satisfies $\omega_i^\flat= \varepsilon_{i}\varepsilon_{i-1} g^\flat \circ N_i$ (see~\cite{AC13}), the quadruple $(g, N_1,N_2, N_3)$ is a (para-)hyperk\"{a}hler structure with torsion on $(A, \mu)$ and its K\"{a}hler forms are $\omega_1, \omega_2$ and $\omega_3$.

Conversely, let us take a  (para-)hyperk\"{a}hler structure with torsion $(\mathtt{g}, I_1, I_2, I_3)$ on $A$ and consider the associated K\"{a}hler forms $\omega_1$, $\omega_2$ and $\omega_3$, given by (\ref{def_Kahler_forms}). We claim that $(\omega_1, \omega_2, \omega_3)$ is a (para-)hypersymplectic structure with torsion on $A$. To prove this, it is enough to show that $I_1, I_2$ and  $I_3$ are the transition morphisms of $(\omega_1, \omega_2, \omega_3)$, defined by (\ref{def_Ni}).
From $\omega_i^\flat= \varepsilon_{i}\varepsilon_{i-1}\mathtt{g}^\flat \circ I_i$ (definition of the K\"{a}hler forms), we get $\mathtt{g}^\flat=\varepsilon_{i-1} \omega_i^\flat \circ I_i$ or, using Lemma~\ref{g_and Kahler},
$$\varepsilon_{i}\omega_i^\flat \circ \pi_{i+1}^{\sharp} \circ \omega_{i-1}^\flat=\omega_i^\flat \circ I_i, $$
which is equivalent to
$$I_i= \varepsilon_{i} \pi_{i+1}^{\sharp} \circ \omega_{i-1}^\flat= \pi_{i-1}^{\sharp} \circ \omega_{i+1}^\flat,$$
where we used $(I_i)^{-1}=\varepsilon_i I_i$ in the last equality.
 \end{proof}

 As a consequence of Theorem~\ref{Thm_1-1_correspondence},
if we pick a (para-)hypersymplectic structure with torsion $(\omega_1, \omega_2, \omega_3)$ on $(A, \mu)$ and consider the (para-)hyperk\"{a}hler structure with torsion $(g, N_1,N_2, N_3)$  given by Theorem~\ref{Thm_1-1_correspondence} then, the (para-)hypersymplectic structure with torsion that corresponds to $(g, N_1,N_2, N_3)$ via Theorem~\ref{Thm_1-1_correspondence}, is the initial one, i.e., $(\omega_1, \omega_2, \omega_3)$.

%
\section{Examples}
\label{sec5}
%
We present, in this section, three examples of hypersymplectic structures with torsion. The first two examples are on Lie algebras, viewed as Lie algebroids over a point, and the third is on the tangent Lie algebroid to a manifold.

The first example is inspired from \cite{GP00}. Let $\{A_1, A_2, B_1, B_2, Z, C_1, C_2, C_3 \}$ be a basis for $\mathbb{R}^8$ and let us consider the Lie algebra structure defined by $$[A_1,B_1]=Z, \quad [A_2,B_2]=-Z$$
and the remaining brackets vanish. We denote by $\{a_1, a_2, b_1, b_2, z, c_1, c_2, c_3 \}$ the dual basis and we define a triplet $(\omega_1, \omega_2, \omega_3)$ of $2$-forms on $\mathbb{R}^8$ by setting
\begin{align*}
    \omega_1&=a_1b_1 + a_2b_2+zc_1+c_2c_3;\\
    \omega_2&=-a_1a_2 +b_1b_2-zc_2+c_1c_3;\\
    \omega_3&=- a_1b_2 + a_2b_1-zc_3-c_1c_2.
\end{align*}
Their matrix representation on the basis $\{A_1, A_2, B_1, B_2, Z, C_1, C_2, C_3 \}$ and its dual is the following:

$${\mathcal {M}}_{\omega_1}=\left(
             \begin{array}{cccccccc}
               0 & 0 & 1 & 0 & 0 & 0 & 0 & 0 \\
              0 & 0 & 0 & 1 & 0 & 0 & 0 & 0 \\
               -1 & 0 & 0 & 0 & 0 & 0 & 0 & 0 \\
               0 & -1 & 0 & 0 & 0 & 0 & 0 & 0 \\
               0 & 0 & 0 & 0 & 0 & 1 & 0 & 0 \\
               0 & 0 & 0 & 0 & -1 & 0 & 0 & 0 \\
               0 & 0 & 0 & 0 & 0 & 0 & 0 & 1 \\
               0 & 0 & 0 & 0 & 0 & 0 & -1 & 0 \\
             \end{array}
           \right),
           {\mathcal {M}}_{\omega_2}=\left(
             \begin{array}{cccccccc}
               0 &-1 & 0 & 0 & 0 & 0 & 0 & 0 \\
               1 & 0 & 0 & 0 & 0 & 0 & 0 & 0 \\
               0 & 0 & 0 & 1 & 0 & 0 & 0 & 0 \\
               0 & 0 & -1 & 0 & 0 & 0 & 0 & 0 \\
               0 & 0 & 0 & 0 & 0 & 0 & -1 & 0 \\
               0 & 0 & 0 & 0 & 0 & 0 & 0 & 1 \\
               0 & 0 & 0 & 0 & 1 & 0 & 0 & 0 \\
               0 & 0 & 0 & 0 & 0 & -1 & 0 & 0 \\
             \end{array}
           \right) $$
           and $$ {\mathcal {M}}_{\omega_3}=\left(
             \begin{array}{cccccccc}
               0 &0 & 0 & -1 & 0 & 0 & 0 & 0 \\
               0 & 0 & 1 & 0 & 0 & 0 & 0 & 0 \\
               0 & -1 & 0 & 0 & 0 & 0 & 0 & 0 \\
               1 & 0 & 0 & 0 & 0 & 0 & 0 & 0 \\
               0 & 0 & 0 & 0 & 0 & 0 & 0 & -1 \\
               0 & 0 & 0 & 0 & 0 & 0 & -1 &0 \\
               0 & 0 & 0 & 0 & 0 & 1 & 0 & 0 \\
               0 & 0 & 0 & 0 & 1 & 0 & 0 & 0 \\
             \end{array}
           \right).$$
           The $2$-forms $\omega_1$, $\omega_2$ and $\omega_3$ are non-degenerate and since
           \begin{equation} \label{d_omega_i}
         {\rm d}\omega_1= -a_1b_1c_1+a_2b_2c_1, \quad    {\rm d}\omega_2= a_1b_1c_2- a_2b_2c_2, \quad  {\rm d}\omega_3= a_1b_1c_3-a_2b_2c_3,
           \end{equation}
they are not closed.
           The transition morphisms $N_1$, $N_2$ and $N_3$, given by (\ref{def_Ni}), correspond to the following matrices in the considered basis:

           $${\mathcal {M}}_{N_{1}}=-{\mathcal {M}}_{\omega_1}, \quad {\mathcal {M}}_{N_{2}}=-{\mathcal {M}}_{\omega_2}, \quad {\mathcal {M}}_{N_{3}}=-{\mathcal {M}}_{\omega_3}.$$
           Using (\ref{d_omega_i}), we have
           \begin{equation*}
    {\rm d}\omega_1(N_1(\cdot),N_1(\cdot),N_1(\cdot))={\rm d}\omega_2(N_2(\cdot),N_2(\cdot),N_2(\cdot))={\rm d}\omega_3(N_3(\cdot),N_3(\cdot),N_3(\cdot))= a_1b_1 z - a_2b_2z,
\end{equation*}
so that the triplet $(\omega_1, \omega_2, \omega_3)$ is a hypersymplectic structure with torsion on $\mathbb{R}^8$. The pseudo-metric $\mathtt{g}$ determined by $(\omega_1,\omega_2,\omega_3)$, defined by (\ref{first_defn _g}), is simply $\mathtt{g}=-{\textrm{id}}_{\mathbb{R}^8}$.

Next, we address an explicit example~\cite{madsen_thesis} of an hypersymplectic structure with torsion on the Lie algebra $\mathfrak{su}(3)$ of the Lie group $SU(3)$.

We write $E_{pq}$ for the elementary $3\times 3$-matrix with $1$ at position $(p,q)$ and consider the basis of $\mathfrak{su}(3)$ consisting of eight complex matrices:
\begin{center}
\begin{tabular}{ll}
  $A_1=i(E_{11}-E_{22})$, &\quad $A_2=i(E_{22}-E_{33})$, \\
  $B_{pq}= E_{pq}-E_{qp}$,&\quad $C_{pq}= i(E_{pq}+E_{qp})$,
\end{tabular}
\end{center}
where $p,q \in \{1,2,3\}$ such that $p<q$. We denote by $\{a_1, a_2, \ldots, c_{23}\}$  the dual basis and
define a triplet of $2$-forms on $SU(3)$ by setting
\begin{align*}
    \omega_1&=-\frac{\sqrt{3}}{2}a_1a_2 + b_{12}c_{12} + b_{13}c_{13}- b_{23}c_{23};\\
    \omega_2&=\frac{\sqrt{3}}{2}a_2b_{12} - a_1c_{12} +\frac{1}{2} a_2c_{12}- b_{13}b_{23} + c_{13}c_{23};\\
    \omega_3&=\frac{\sqrt{3}}{2}a_2c_{12} + a_1b_{12} -\frac{1}{2} a_2b_{12}+ b_{13}c_{23} + b_{23}c_{13}.
\end{align*}
The $2$-forms have a matrix representation, on the basis $(A_1, A_2, B_{12}, B_{13}, B_{23}, C_{12}, C_{13}, C_{23})$ and its dual, given by
$${\mathcal {M}}_{\omega_1}=\left(
             \begin{array}{cccccccc}
               0 & -\frac{\sqrt{3}}{2} & 0 & 0 & 0 & 0 & 0 & 0 \\
               \frac{\sqrt{3}}{2} & 0 & 0 & 0 & 0 & 0 & 0 & 0 \\
               0 & 0 & 0 & 0 & 0 & 1 & 0 & 0 \\
               0 & 0 & 0 & 0 & 0 & 0 & 1 & 0 \\
               0 & 0 & 0 & 0 & 0 & 0 & 0 & -1 \\
               0 & 0 & -1 & 0 & 0 & 0 & 0 & 0 \\
               0 & 0 & 0 & -1 & 0 & 0 & 0 & 0 \\
               0 & 0 & 0 & 0 & 1 & 0 & 0 & 0 \\
             \end{array}
           \right),
{\mathcal {M}}_{\omega_2}=\left(
             \begin{array}{cccccccc}
               0 & 0 & 0 & 0 & 0 & -1 & 0 & 0 \\
               0 & 0 & \frac{\sqrt{3}}{2} & 0 & 0 & \frac{1}{2} & 0 & 0 \\
               0 & -\frac{\sqrt{3}}{2} & 0 & 0 & 0 & 0 & 0 & 0 \\
               0 & 0 & 0 & 0 & -1 & 0 & 0 & 0 \\
               0 & 0 & 0 & 1 & 0 & 0 & 0 & 0 \\
               1 & -\frac{1}{2} & 0 & 0 & 0 & 0 & 0 & 0 \\
               0 & 0 & 0 & 0 & 0 & 0 & 0 & -1 \\
               0 & 0 & 0 & 0 & 0 & 0 & 1 & 0 \\
             \end{array}
           \right)$$
           and
$${\mathcal {M}}_{\omega_3}=\left(
             \begin{array}{cccccccc}
               0 & 0 & 1 & 0 & 0 & 0 & 0 & 0 \\
               0 & 0 & -\frac{1}{2} & 0 & 0 & \frac{\sqrt{3}}{2} & 0 & 0 \\
               -1 & \frac{1}{2} & 0 & 0 & 0 & 0 & 0 & 0 \\
               0 & 0 & 0 & 0 & 0 & 0 & 0 & 1 \\
               0 & 0 & 0 & 0 & 0 & 0 & 1 & 0 \\
               0 & -\frac{\sqrt{3}}{2} & 0 & 0 & 0 & 0 & 0 & 0 \\
               0 & 0 & 0 & 0 & -1 & 0 & 0 & 0 \\
               0 & 0 & 0 & -1 & 0 & 0 & 0 & 0 \\
             \end{array}
           \right).$$
These $2$-forms are not closed; for example, we have
\begin{multline*}
{\rm d}\omega_1=-\sqrt{3}a_1(b_{13}c_{13}+b_{23}c_{23})+ \sqrt{3}a_2(b_{12}c_{12}\\+b_{13}c_{13})-b_{12}b_{13}c_{23}-b_{12}b_{23}c_{13}-b_{13}b_{23}c_{12}-c_{12}c_{13}c_{23}.
\end{multline*}
The transition morphisms $N_1$, $N_2$ and $N_3$, given by (\ref{def_Ni}), correspond to the following matrices in the considered basis:
$${\mathcal {M}}_{N_{1}}=\left(
             \begin{array}{cccccccc}
               -\frac{1}{\sqrt{3}} & \frac{2}{\sqrt{3}} & 0 & 0 & 0 & 0 & 0 & 0 \\
               -\frac{2}{\sqrt{3}} & \frac{1}{\sqrt{3}} & 0 & 0 & 0 & 0 & 0 & 0 \\
               0 & 0 & 0 & 0 & 0 & -1 & 0 & 0 \\
               0 & 0 & 0 & 0 & 0 & 0 & -1 & 0 \\
               0 & 0 & 0 & 0 & 0 & 0 & 0 & 1 \\
               0 & 0 & 1 & 0 & 0 & 0 & 0 & 0 \\
               0 & 0 & 0 & 1 & 0 & 0 & 0 & 0 \\
               0 & 0 & 0 & 0 & -1 & 0 & 0 & 0 \\
             \end{array}
           \right),
{\mathcal {M}}_{N_{2}}=\left(
             \begin{array}{cccccccc}
               0 & 0 & -\frac{1}{\sqrt{3}} & 0 & 0 & 1 & 0 & 0 \\
               0 & 0 & -\frac{2}{\sqrt{3}} & 0 & 0 & 0 & 0 & 0 \\
               0 & \frac{\sqrt{3}}{2} & 0 & 0 & 0 & 0 & 0 & 0 \\
               0 & 0 & 0 & 0 & 1 & 0 & 0 & 0 \\
               0 & 0 & 0 & -1 & 0 & 0 & 0 & 0 \\
               -1 & \frac{1}{2} & 0 & 0 & 0 & 0 & 0 & 0 \\
               0 & 0 & 0 & 0 & 0 & 0 & 0 & 1 \\
               0 & 0 & 0 & 0 & 0 & 0 & -1 & 0 \\
             \end{array}
           \right)$$
           and
$${\mathcal {M}}_{N_{3}}=\left(
             \begin{array}{cccccccc}
               0 & 0 & -1 & 0 & 0 & -\frac{1}{\sqrt{3}} & 0 & 0 \\
               0 & 0 & 0 & 0 & 0 & -\frac{2}{\sqrt{3}} & 0 & 0 \\
               1 & -\frac{1}{2} & 0 & 0 & 0 & 0 & 0 & 0 \\
               0 & 0 & 0 & 0 & 0 & 0 & 0 & -1 \\
               0 & 0 & 0 & 0 & 0 & 0 & -1 &  \\
               0 & \frac{\sqrt{3}}{2} & 0 & 0 & 0 & 0 & 0 & 0 \\
               0 & 0 & 0 & 0 & 1 & 0 & 0 & 0 \\
               0 & 0 & 0 & 1 & 0 & 0 & 0 & 0 \\
             \end{array}
           \right).$$
An easy computation gives
\begin{align*}
    &{\rm d}\omega_1(N_1(\cdot),N_1(\cdot),N_1(\cdot))={\rm d}\omega_2(N_2(\cdot),N_2(\cdot),N_2(\cdot))={\rm d}\omega_3(N_3(\cdot),N_3(\cdot),N_3(\cdot))\\
    &\quad= -a_1b_{13}c_{13}+a_1b_{23}c_{23} -2a_1b_{12}c_{12} - a_2b_{13}c_{13} - 2a_2b_{23}c_{23}\\
    &\quad\quad + a_2b_{12}c_{12} +b_{23}c_{12}c_{13} + b_{13}c_{12}c_{23} + b_{12}c_{13}c_{23} + b_{12}b_{13}b_{23},
\end{align*}
which shows that the triplet $(\omega_1, \omega_2, \omega_3)$ is a hypersymplectic structure with torsion on $\mathfrak{su}(3)$.
Finally, the pseudo-metric is given by
$$
{\mathcal {M}}_{\mathtt{g}}
=\left(
             \begin{array}{cccccccc}
               -1 & \frac{1}{2} & 0 & 0 & 0 & 0 & 0 & 0 \\
               \frac{1}{2} & -1 & 0 & 0 & 0 & 0 & 0 & 0 \\
               0 & 0 & -1 & 0 & 0 & 0 & 0 & 0 \\
               0 & 0 & 0 & -1 & 0 & 0 & 0 & 0 \\
               0 & 0 & 0 & 0 & -1 & 0 & 0 & 0 \\
               0 & 0 & 0 & 0 & 0 & -1 & 0 & 0 \\
               0 & 0 & 0 & 0 & 0 & 0 & -1 & 0 \\
               0 & 0 & 0 & 0 & 0 & 0 & 0 & -1 \\
             \end{array}
           \right).
$$

In the third example, which is taken form \cite{CS08}, we describe a hypersymplectic structure on the Lie algebroid tangent to the manifold $M=S^3 \times (S^1)^5$. The sphere $S^3$ is identified with the Lie group $Sp(1)$. In its Lie algebra $\mathfrak{sp}(1)$ we consider a basis $\{A_2, A_3, A_4 \}$ and the brackets
$$[A_2, A_3]=2A_4, \quad [A_3, A_4]=2A_2, \quad [A_4, A_2]=2A_3.$$
Let $\{a_2, a_3, a_4 \}$ be the dual basis and let us consider a basis $\{a_1, a_5, a_6, a_7, a_8 \}$ of $1$-forms on $(S^1)^5$. We define a triplet $(\omega_1, \omega_2, \omega_3)$ of $2$-forms on $M$ by setting, on the basis $\{a_2, a_3, a_4, a_1, a_5, a_6, a_7, a_8 \}$,
\begin{align*}
    \omega_1&=a_2a_1 +a_4a_3+a_6a_5+a_8a_7;\\
    \omega_2&=a_3a_1+a_2a_4+a_7a_5+a_6a_8;\\
    \omega_3&=a_4a_1+a_3a_2+a_8a_5+a_7a_6.
\end{align*}
Their matrix representation on the considered basis is given by
$${\mathcal {M}}_{\omega_1}=\left(
             \begin{array}{cccccccc}
               0 & 0 & 0 & 1 & 0 & 0 & 0 & 0 \\
              0 & 0 & -1 & 0 & 0 & 0 & 0 & 0 \\
               0 & 1 & 0 & 0 & 0 & 0 & 0 & 0 \\
               -1 & 0 & 0 & 0 & 0 & 0 & 0 & 0 \\
               0 & 0 & 0 & 0 & 0 & -1 & 0 & 0 \\
               0 & 0 & 0 & 0 & 1 & 0 & 0 & 0 \\
               0 & 0 & 0 & 0 & 0 & 0 & 0 & -1 \\
               0 & 0 & 0 & 0 & 0 & 0 & 1 & 0 \\
             \end{array}
           \right),
          {\mathcal {M}}_{\omega_2}=\left(
             \begin{array}{cccccccc}
               0 &0 & 1 & 0 & 0 & 0 & 0 & 0 \\
               0 & 0 & 0 & 1 & 0 & 0 & 0 & 0 \\
               -1 & 0 & 0 & 0 & 0 & 0 & 0 & 0 \\
               0 & -1 & 0 & 0 & 0 & 0 & 0 & 0 \\
               0 & 0 & 0 & 0 & 0 & 0 & -1 & 0 \\
               0 & 0 & 0 & 0 & 0 & 0 & 0 & 1 \\
               0 & 0 & 0 & 0 & 1 & 0 & 0 & 0 \\
               0 & 0 & 0 & 0 & 0 & -1 & 0 & 0 \\
             \end{array}
           \right) $$
           and $$ {\mathcal {M}}_{\omega_3}=\left(
             \begin{array}{cccccccc}
               0 &-1 & 0 & 0 & 0 & 0 & 0 & 0 \\
               1 & 0 & 0 & 0 & 0 & 0 & 0 & 0 \\
               0 & 0 & 0 & 1 & 0 & 0 & 0 & 0 \\
               0 & 0 & -1 & 0 & 0 & 0 & 0 & 0 \\
               0 & 0 & 0 & 0 & 0 & 0 & 0 & -1 \\
               0 & 0 & 0 & 0 & 0 & 0 & -1 &0 \\
               0 & 0 & 0 & 0 & 0 & 1 & 0 & 0 \\
               0 & 0 & 0 & 0 & 1 & 0 & 0 & 0 \\
             \end{array}
           \right).$$
           The $2$-forms are non-degenerate and not closed:
           \begin{equation*}
           \d \omega_1= -2 a_1a_3a_4, \quad \d \omega_2= -2a_1a_4a_2, \quad \d \omega_3= -2a_1a_2a_3.
           \end{equation*}
 The transition morphisms $N_1$, $N_2$ and $N_3$, given by (\ref{def_Ni}), correspond to the following matrices:
        $${\mathcal {M}}_{N_{1}}={\mathcal {M}}_{\omega_1}, \quad {\mathcal {M}}_{N_{2}}={\mathcal {M}}_{\omega_2}, \quad {\mathcal {M}}_{N_{3}}={\mathcal {M}}_{\omega_3}.$$
Moreover,       $$
    {\rm d}\omega_1(N_1(\cdot),N_1(\cdot),N_1(\cdot))={\rm d}\omega_2(N_2(\cdot),N_2(\cdot),N_2(\cdot))={\rm d}\omega_3(N_3(\cdot),N_3(\cdot),N_3(\cdot))=2a_2a_3a_4,
$$
which shows that the triplet $(\omega_1, \omega_2, \omega_3)$ is a hypersymplectic structure with torsion on the Lie algebroid $TM$. Regarding the pseudo-metric, we have $\mathtt{g}={\rm id}_{TM}$.

%
\section{The pre-Courant algebroid case}
\label{sec6}
%

In this section we firstly recall the notion of \emph{(para-)hypersymplectic structure} on a pre-Courant algebroid, introduced in~\cite{AC14a}. Then, we prove some results involving the Nijenhuis torsions of morphisms on $A$ and on $A\oplus A^*$.

In order to simplify the notation, when $\mathcal{I}$ and $\mathcal{J}$ are endomorphisms of a pre-Courant algebroid, the composition $\mathcal{I}\circ \mathcal{J}$ will be denoted by $\mathcal{I}\mathcal{J}$. 

\

Let us consider a triplet $(\mathcal{S}_1, \mathcal{S}_2, \mathcal{S}_3)$ of skew-symmetric (w.r.t. $\langle\cdot,\cdot\rangle$) endomorphisms $\mathcal{S}_i: E \to E$, $i=1,2,3$, on a pre-Courant algebroid $(E, \langle\cdot,\cdot\rangle, \Theta)$, such that
\begin{enumerate}
  \item ${\mathcal{S}_i}^2=\varepsilon_i \,{\rm id}_E$,
  \item $\mathcal{S}_i\mathcal{S}_j=-\mathcal{S}_j\mathcal{S}_i$,\ \ $i,j \in \{1,2,3\},\ i\neq j$,
  \item $\Theta_{\mathcal{S}_i,\mathcal{S}_i}=\varepsilon_i \Theta$, \footnote{We use the following notation: $\Theta_{\mathcal{I}}=\{\mathcal{I}, \Theta \}$ and $\Theta_{\mathcal{I}, \mathcal{J}}=\{\mathcal{J},\{\mathcal{I}, \Theta \}\}$, with $\mathcal{I}, \mathcal{J}$ skew-symmetric endomorphisms of $E$. For further details on the supergeometric approach see~\cite{royContemp}.}
\end{enumerate}
where the parameters $\varepsilon_i=\pm 1$ satisfy $\varepsilon_1\varepsilon_2\varepsilon_3=-1$.

%
If $\varepsilon_1=\varepsilon_2=\varepsilon_3=-1$, then $(\mathcal{S}_1, \mathcal{S}_2, \mathcal{S}_3)$ is said to be a \emph{hypersymplectic} structure on $(E, \Theta)$.
Otherwise, $(\mathcal{S}_1, \mathcal{S}_2, \mathcal{S}_3)$ is said to be a \emph{para-hypersymplectic} structure on $(E, \Theta)$.




The \emph{transition morphisms} are the endomorphisms $\mathcal{T}_1,\mathcal{T}_2$ and $\mathcal{T}_3$ of $E$ defined as
\begin{equation} \label{def_transition}
\mathcal{T}_i:=\varepsilon_{i-1} \mathcal{S}_{i-1}\mathcal{S}_{i+1}, \quad\quad i \in \mathbb{Z}_3.
\end{equation}
\

Among the pre-Courant algebroid structures, we shall be interested in those defined on vector bundles of type $A\oplus A^*$, equipped with the natural pairing, since these can be related to structures on $A$.

If we take a triplet $(\omega_1, \omega_2, \omega_3)$ of $2$-forms and a triplet $(\pi_1,\pi_2, \pi_3)$ of bivectors on $A$, we may define the skew-symmetric endomorphisms \mbox{$\mathcal{S}_i: A \oplus A^* \to A \oplus A^*$}, $i=1,2,3$,
\begin{equation} \label{definition_Si}
\mathcal{S}_i:=\left(
                    \begin{array}{ccc}
                    0&\ &\varepsilon_i\,\pi_{i}^{\sharp}\\
                    \omega_{i}^{\flat}&\ &0
                    \end{array}
\right).
\end{equation}

%
The next proposition was proved in \cite{AC14a} for the hypersymplectic case. The para-hypersymplectic case has an analogous proof.

\begin{prop}  \label{HS_structure_mu+gamma+psi}
Let $(\omega_1, \omega_2, \omega_3)$ be a triplet of $2$-forms and $(\pi_1,\pi_2, \pi_3)$ be a triplet  of bivectors on a Lie algebroid $(A,\mu)$. Consider the triplet $(\mathcal{S}_1, \mathcal{S}_2, \mathcal{S}_3)$ of endomorphisms of $A \oplus A^*$, with $\mathcal{S}_i$ given by (\ref{definition_Si}). The following assertions are equivalent:
\begin{enumerate}
    \item $(\omega_1, \omega_2, \omega_3)$ is a (para-)hypersymplectic structure with torsion on the Lie algebroid $(A,\mu)$ and $\pi_i$ is the inverse of $\omega_i$, $i=1,2,3$;
    \item $(\mathcal{S}_1, \mathcal{S}_2, \mathcal{S}_3)$ is a (para-)hypersymplectic structure on the pre-Courant algebroid $(A\oplus A^*, \mu+\psi)$, for some $\psi\in \Gamma(\wedge^3 A)$.
\end{enumerate}
\end{prop}

Notice that in the assertion ii) of Proposition~\ref{HS_structure_mu+gamma+psi}, since $(\mathcal{S}_1, \mathcal{S}_2, \mathcal{S}_3)$ is a hypersymplectic structure, condition $(\mu+\psi)_{\mathcal{S}_k, \mathcal{S}_k}= \varepsilon_k(\mu+\psi)$ holds and implies that $\psi$ has to be of the form $\frac{\varepsilon_k}{2}[\pi_k, \pi_k]$, for any $k\in \{1,2,3\}$.

Under the conditions of Proposition~\ref{HS_structure_mu+gamma+psi}, the transition morphisms of the  (para-)hypersymplectic structure $(\mathcal{S}_1, \mathcal{S}_2, \mathcal{S}_3)$ on $(A\oplus A^*,\mu+\psi)$, defined by  (\ref{def_transition}), are given by
\begin{equation} \label{def_T_i}
\mathcal{T}_i=\left(
                    \begin{array}{ccc}
                        N_i&\ &0\\
                        0&\ & -{N_i}^*
                    \end{array}
                \right), \quad i=1,2,3,
\end{equation}
where $N_i$ is defined by (\ref{def_Ni}).
%

Recall that the Nijenhuis torsion ${\text{\Fontlukas T}}_{\mu}N$ of an endomorphism $N$ on a pre-Lie algebroid\,\footnote{A pre-Lie algebroid is a pair $(A, \mu)$ that satisfies the axioms of the Lie algebroid definition except, eventually, the Jacobi identity. In other words, we may have $\bb{\mu}{\mu}\neq 0$.} $(A, \mu)$ is given by
\begin{equation} \label{torsion_N}
{\text{\Fontlukas T}}_{\mu}N (X,Y)=[NX,NY]- N([NX,Y]+[X,NY]-N[X,Y]),
\end{equation}
for all $X,Y \in \Gamma(A)$.

The Dorfman bracket of a pre-Courant algebroid $(E, \Theta)$ is defined, for all $X,Y \in \Gamma(E)$, by the derived bracket expression
\begin{equation}\label{eq_derived_bracket_expressions}
  {\Dorf{X}{Y}=\{\{X,\Theta\},Y\}}.
\end{equation}
The Nijenhuis torsion ${\text{\Fontlukas T}}_{\Theta}{\mathcal{I}}$ of an endomorphism $\mathcal{I}$ on $(E, \Theta)$ is given by~\cite{ALC11,YKS11}
\begin{equation}\label{def_Nijenhuis_torsion}
    {\text{\Fontlukas T}}_{\Theta}{\mathcal{I}}(X,Y)=\Dorf{\mathcal{I}X}{\mathcal{I}Y}-\mathcal{I}\left(\Dorf{X}{Y}_{\mathcal{I}}\right)
        =\frac{1}{2}\left(\Dorf{X}{Y}_{\mathcal{I},\mathcal{I}}-\Dorf{X}{Y}_{\mathcal{I}^2}\right),
\end{equation}
for all $X,Y \in \Gamma(E)$, where
\begin{equation} \label{deformed_Dorfman_bracket}
    {\Dorf{X}{Y}_{\mathcal{I}}} =\Dorf{{\mathcal{I}}X}{Y}+\Dorf{X}{{\mathcal{I}}Y}-{\mathcal{I}}\Dorf{X}{Y}
    \quad\textrm{ and }\quad \Dorf{\cdot}{\cdot}_{\mathcal{I},\mathcal{I}}=\Big(\Dorf{\cdot}{\cdot}_{\mathcal{I}}\Big)_\mathcal{I}.
\end{equation}

 The next proposition addresses a relation between ${\text{\Fontlukas T}}_{\mu}N$ and the Nijenhuis torsion of the skew-symmetric
 morphism $\mathcal{D}_{\!{}_N}=N\oplus (-N^*)$ on the pre-Courant algebroid $(A\oplus A^*, \mu)$ (see (\ref{def_Nijenhuis_torsion})). To the best of our knowledge formula (\ref{eq_relation_T(JN)_and_T(N)}) is new.


\begin{prop}\label{prop_relation_T(JN)_and_T(N)}
    Let $(A,\mu)$ be a pre-Lie algebroid and define $\mathcal{D}_{\!{}_N}:A\oplus A^* \to A\oplus A^*$ by setting $\mathcal{D}_{\!{}_N}=\left(
                                                                                                          \begin{array}{cc}
                                                                                                            N & 0 \\
                                                                                                            0 & -N^* \\
                                                                                                          \end{array}
                                                                                                        \right)$,
    where $N: A\to A$ is a bundle morphism. The Nijenhuis torsion ${\text{\Fontlukas T}}_{\mu}\mathcal{D}_{\!{}_N}$ of the endomorphism $\mathcal{D}_{\!{}_N}$ on the pre-Courant algebroid $(A\oplus A^*,\mu)$ is given by
\begin{multline}\label{eq_relation_T(JN)_and_T(N)}
    {\text{\Fontlukas T}}_{\mu}\mathcal{D}_{\!{}_N}(X+\alpha, Y+\beta)=\Dorf{X+\alpha}{Y+\beta}_{{\text{\Fontlukas T}}_{\mu}N}\\+(N^*)^2\Dorf{X}{ \beta}-\Dorf{X}{(N^*)^2\beta}
            +(N^*)^2\Dorf{\alpha}{Y}-\Dorf{(N^*)^2\alpha}{Y},
\end{multline}
for all $X+\alpha, Y+\beta \in \Gamma(A\oplus A^*)$, where $\Dorf{\cdot}{\cdot}$ and $\Dorf{\cdot}{\cdot}_{{\text{\Fontlukas T}}_{\mu}N}$ stand for the Dorfman brackets determined by the superfunctions $\mu$ and ${\text{\Fontlukas T}}_{\mu}N$, respectively\footnote{Note that ${\text{\Fontlukas T}}_{\mu}N$ is seen as a function in the supergeometric sense, so that we may consider the Dorfman bracket associated to it.}, according to (\ref{eq_derived_bracket_expressions}).
\end{prop}
\begin{proof}
Using the $\mathbb{R}$-bilinearity of ${\text{\Fontlukas T}}_{\mu}\mathcal{D}_{\!{}_N}$ we have
\begin{multline}\label{eq_aux1_proof_relation_T(JN)_and_T(N)}
    {\text{\Fontlukas T}}_{\mu}\mathcal{D}_{\!{}_N}(X+\alpha, Y+\beta)={\text{\Fontlukas T}}_{\mu}\mathcal{D}_{\!{}_N}(X, Y)+{\text{\Fontlukas T}}_{\mu}\mathcal{D}_{\!{}_N}(X, \beta)\\+{\text{\Fontlukas T}}_{\mu}\mathcal{D}_{\!{}_N}(\alpha, Y)+{\text{\Fontlukas T}}_{\mu}\mathcal{D}_{\!{}_N}(\alpha, \beta),
\end{multline}
for all $X,Y \in \Gamma(A)$ and $\alpha, \beta \in \Gamma(A^*)$, where we identify $X+0 \in \Gamma(A \oplus A^*)$ with $X$ and $0+ \alpha \in \Gamma(A \oplus A^*)$ with $\alpha$. In what follows, we explicit each of the summands of the r.h.s. of Equation (\ref{eq_aux1_proof_relation_T(JN)_and_T(N)}).
For the first summand, denoting by $\Dorf{\cdot}{\cdot}_{{\mathcal D}_N}$ the deformation of the Dorfman bracket $\Dorf{\cdot}{\cdot}$ by ${\mathcal D}_N$ (see (\ref{deformed_Dorfman_bracket})), we have
\begin{align}
    {\text{\Fontlukas T}}_{\mu}\mathcal{D}_{\!{}_N}(X, Y)&=\Dorf{\mathcal{D}_{\!{}_N}X}{\mathcal{D}_{\!{}_N}Y}-\mathcal{D}_{\!{}_N}\left(\Dorf{X}{Y}_{\mathcal{D}_{\!{}_N}}\right)\nonumber\\
    &=\Dorf{\mathcal{D}_{\!{}_N}X}{\mathcal{D}_{\!{}_N}Y}-\mathcal{D}_{\!{}_N}\Dorf{\mathcal{D}_{\!{}_N}X}{Y} - \mathcal{D}_{\!{}_N}\Dorf{X}{\mathcal{D}_{\!{}_N}Y} + \mathcal{D}_{\!{}_N}^2\Dorf{X}{Y}\nonumber\\
    &=[NX, NY]-N[NX,Y] - N[X,NY] + N^2[X,Y]\nonumber\\
    &={\text{\Fontlukas T}}_{\mu}N(X, Y),\label{eq_aux2_proof_relation_T(JN)_and_T(N)}
\end{align}
where we used (\ref{def_Nijenhuis_torsion}) and (\ref{torsion_N}) and the fact that, when restricted to sections of $A$, $\mathcal{D}_{\!{}_N}$ coincides with $N$. In the supergeometric setting, (\ref{eq_aux2_proof_relation_T(JN)_and_T(N)}) may be written as
\begin{equation}\label{eq_aux3_proof_relation_T(JN)_and_T(N)}
    {\text{\Fontlukas T}}_{\mu}\mathcal{D}_{\!{}_N}(X, Y)=\frac{1}{2}\bb{\bb{X}{\mu_{N,N}-\mu_{N^2}}}{Y}.
\end{equation}
The second summand of the r.h.s. of (\ref{eq_aux1_proof_relation_T(JN)_and_T(N)}) can be written as (see (\ref{def_Nijenhuis_torsion}))
\begin{equation}\label{eq_aux4_proof_relation_T(JN)_and_T(N)}
    {\text{\Fontlukas T}}_{\mu}\mathcal{D}_{\!{}_N}(X, \beta)=\frac{1}{2}\left(\Dorf{X}{\beta}_{\mathcal{D}_{\!{}_N},\mathcal{D}_{\!{}_N}}-\Dorf{X}{\beta}_{\mathcal{D}_{\!{}_N}^2}\right).
\end{equation}
The key argument of this proof is the relation between morphisms $\mathcal{D}_{\!{}_N}^2= \mathcal{D}_N \circ \mathcal{D}_N$ and $\mathcal{D}_{\!{}_{N^2}}$. In fact
$$\mathcal{D}_{\!{}_N}^2(X + \alpha)=N^2(X)+ {N^*}^2(\alpha), \quad \text{while}\quad \mathcal{D}_{\!{}_{N^2}}(X + \alpha)=N^2(X) - {N^*}^2(\alpha).$$
Thus, Equation (\ref{eq_aux4_proof_relation_T(JN)_and_T(N)}) becomes
\begin{align}\label{eq_aux5_proof_relation_T(JN)_and_T(N)}
    {\text{\Fontlukas T}}_{\mu}\mathcal{D}_{\!{}_N}(X, \beta)&=\frac{1}{2}\left(\Dorf{X}{\beta}_{\mathcal{D}_{\!{}_N},\mathcal{D}_{\!{}_N}}-\Dorf{X}{\beta}_{\mathcal{D}_{\!{}_{N^2}}}\right)-\Dorf{X}{ {N^*}^2\beta}+{N^*}^2\Dorf{X}{\beta}\nonumber\\
    &=\frac{1}{2}\bb{\bb{X}{\mu_{N,N}-\mu_{N^2}}}{\beta}-\Dorf{X}{{N^*}^2\beta}+{N^*}^2\Dorf{X}{\beta}.
\end{align}
Analogously, the third summand of the r.h.s. of Equation (\ref{eq_aux1_proof_relation_T(JN)_and_T(N)}) is given by
\begin{equation}\label{eq_aux6_proof_relation_T(JN)_and_T(N)}
    {\text{\Fontlukas T}}_{\mu}\mathcal{D}_{\!{}_N}(\alpha, Y)=\frac{1}{2}\bb{\bb{\alpha}{\mu_{N,N}-\mu_{N^2}}}{Y}-\Dorf{{N^*}^2\alpha}{ Y}+{N^*}^2\Dorf{\alpha}{Y}.
\end{equation}
Finally, the fourth summand vanishes because the Dorfman bracket vanishes when restricted to sections of $\Gamma(A^*)$. Thus, using (\ref{eq_aux3_proof_relation_T(JN)_and_T(N)}), (\ref{eq_aux5_proof_relation_T(JN)_and_T(N)}) and (\ref{eq_aux6_proof_relation_T(JN)_and_T(N)}), Equation (\ref{eq_aux1_proof_relation_T(JN)_and_T(N)}) becomes
\begin{align*}
    {\text{\Fontlukas T}}_{\mu}\mathcal{D}_{\!{}_N}(X+\alpha, Y+\beta)&=\bb{\bb{X+\alpha}{\frac{1}{2}\left(\mu_{N,N}-\mu_{N^2}\right)}}{Y+\beta}\\
        &\quad -\Dorf{X}{{N^*}^2\beta}+{N^*}^2\Dorf{X}{\beta}-\Dorf{{N^*}^2\alpha}{Y}+{N^*}^2\Dorf{\alpha}{Y}\\
        &=\Dorf{X+\alpha}{Y+\beta}_{{\text{\Fontlukas T}}_{\mu}N}\\
        &\quad -\Dorf{X}{{N^*}^2\beta}+{N^*}^2\Dorf{X}{\beta}-\Dorf{{N^*}^2\alpha}{Y}+{N^*}^2\Dorf{\alpha}{Y}.
\end{align*}
\end{proof}

\begin{cor}\label{cor_T(JN)=0->T(N)=0}
If $\mathcal{D}_{\!{}_N}$ is a Nijenhuis morphism on a pre-Courant algebroid $(A \oplus A^*, \mu)$, then $N$ is a Nijenhuis morphism on the pre-Lie algebroid $(A, \mu)$.
\end{cor}
\begin{proof}
    It suffices to evaluate (\ref{eq_relation_T(JN)_and_T(N)}) on pairs of type $(X+0, Y+0)$.
\end{proof}

The next proposition, that already appears in~\cite{YKS11}, is a direct consequence of (\ref{eq_relation_T(JN)_and_T(N)}). The notations used are the same as in Proposition \ref{prop_relation_T(JN)_and_T(N)}.

\begin{prop} \label{torsion_TN_torsion_N}
    Let $N: A\to A$ be a bundle morphism such that $N^2=\lambda {\rm id}_A$, for some $\lambda \in \mathbb{R}$. Then, the skew-symmetric morphism $\mathcal{D}_{\!{}_N}=N \oplus (-N^*)$ satisfies $\mathcal{D}_{\!{}_N}^2=\lambda {\rm id}_{A\oplus A^*}$ and, in this case, ${\text{\Fontlukas T}}_{\mu}\mathcal{D}_{\!{}_N}=0$ if and only if ${\text{\Fontlukas T}}_{\mu}N=0$.
\end{prop}

Let us recall that the concomitant $C_\Theta({\mathcal{I}},{\mathcal{J}})$ of two skew-symmetric endomorphisms $\mathcal{I}$ and $\mathcal{J}$ of a pre-Courant algebroid $(E, \Theta)$ is defined by~\cite{ALC11}:
\begin{equation}  \label{def_conc}
C_\Theta({\mathcal{I}},{\mathcal{J}}):=\Theta _{{\mathcal{I}},{\mathcal{J}}}+\Theta_{{\mathcal{J}},{\mathcal{I}}}.
\end{equation}

The next proposition establishes a relation between the Fr\"{o}licher-Nijenhuis bracket $[I,J]_{FN}$ of two endomorphisms on a pre-Lie algebroid $(A, \mu)$ and the concomitant $C_{\mu}(\mathcal{D}_{\!{}_{I}},\mathcal{D}_{\!{}_{J}})$ of the induced morphisms on $(A\oplus A^*, \mu)$ and it will be used in the next section.
%


\begin{prop}\label{prop_C_mu(Ti,Tj)=-2[I,J]_FN}
    Let $(A,\mu)$ be a pre-Lie algebroid and $I, J: A \to A$ two anticommuting endomorphisms of $A$. Then, $\mathcal{D}_{\!{}_{I}}:=\left(
                                                                                                          \begin{array}{cc}
                                                                                                            I & 0 \\
                                                                                                            0 & -I^* \\
                                                                                                          \end{array}
                                                                                                        \right)$
    and $\mathcal{D}_{\!{}_J}:=\left(
                    \begin{array}{cc}
                         J & 0 \\
                         0 & -J^* \\
                    \end{array}
              \right)$  are anticommuting skew-symmetric endomorphisms of $A\oplus A^*$ and
        $$[I,J]_{FN}(X,Y)=-\frac{1}{2}C_{\mu}(\mathcal{D}_{\!{}_{I}},\mathcal{D}_{\!{}_{J}})(X+0, Y+0),$$
    for all $X,Y \in \Gamma(A)$. In particular, if $C_{\mu}(\mathcal{D}_{\!{}_{I}},\mathcal{D}_{\!{}_{J}})$ vanishes then so does $[I,J]_{FN}$.
\end{prop}
\begin{proof}
    The fact that $\mathcal{D}_{\!{}_{I}}$ and $\mathcal{D}_{\!{}_{J}}$ are anticommuting skew-symmetric endomorphisms of $A\oplus A^*$ is immediate to check. Moreover, using formula (5) in~\cite{AC13}, %
     we have
\begin{equation}\label{eq_aux1_proof_C_mu(Ti,Tj)=-2[I,J]_FN}
    [I,J]_{FN}=\bb{\bb{I}{\mu}}{J}+\bb{I\circ J}{\mu}.
\end{equation}
Because $I$ and $J$ anticommute, $I\circ J=\frac{1}{2}\bb{J}{I}$ and (\ref{eq_aux1_proof_C_mu(Ti,Tj)=-2[I,J]_FN}) becomes
\begin{equation*}
    [I,J]_{FN}=\bb{\bb{I}{\mu}}{J}+\frac{1}{2}\bb{\bb{J}{I}}{\mu}.
\end{equation*}
Using the Jacobi identity of the big bracket, we get
    $$[I,J]_{FN}=-\frac{1}{2}\big(\bb{J}{\bb{I}{\mu}}+\bb{I}{\bb{J}{\mu}}\big)=-\frac{1}{2}C_{\mu}(\mathcal{D}_{\!{}_{I}},\mathcal{D}_{\!{}_{J}}).$$
\end{proof}

%
\section{Compatibilities and deformations}
\label{sec7}
%

It was proved in \cite{CS08} that, when $\varepsilon_1=\varepsilon_2=\varepsilon_3=-1$, condition (\ref{HKT_axioms})v) in Definition~\ref{HKT_definition} implies that the Nijenhuis torsion of the endomorphisms $I_1, I_2$ and $I_3$ vanishes, so that they are in fact complex structures on $(A, \mu)$. Taking into account Theorem \ref{Thm_1-1_correspondence}, the next theorem can be seen as a generalization of the mentioned result in \cite{CS08}.

\begin{thm}\label{Thm_PHST->Ni_Nijenhuis}
    Let $(\omega_1, \omega_2, \omega_3)$ be a (para-)hypersymplectic structure with torsion on a Lie algebroid $(A, \mu)$. The endomorphisms $N_1,N_2$ and $N_3$ given by (\ref{def_Ni}) are Nijenhuis morphisms.
\end{thm}
\begin{proof}
    As a consequence of Proposition~\ref{HS_structure_mu+gamma+psi}, the triplet $(\mathcal{S}_1, \mathcal{S}_2, \mathcal{S}_3)$, with $\mathcal{S}_i$ given by (\ref{definition_Si}), is a (para-)hypersymplectic structure on the pre-Courant algebroid $(A\oplus A^*, \mu+\psi)$, with $\psi= \frac{\varepsilon_k}{2}[\pi_k, \pi_k]$, for any $k \in \{1,2,3\}$. Using Theorem 5.3 in~\cite{AC14a}, the endomorphisms $\mathcal{T}_i$, $i=1,2,3,$ given by (\ref{def_T_i}) are Nijenhuis morphisms on $(A\oplus A^*, \mu+\psi)$. This means that, using Corollary 3 in~\cite{grab},
$$\left\{
    \begin{array}{l}
      \mathcal{T}_i^2=\varepsilon_i {\rm id}_{A\oplus A^*}\\
      (\mu+\psi)_{\mathcal{T}_i,\mathcal{T}_i}=\varepsilon_i (\mu+\psi).
    \end{array}
  \right.
$$
Splitting up the equality $(\mu+\psi)_{\mathcal{T}_i,\mathcal{T}_i}=\varepsilon_i (\mu+\psi)$ in terms of bidegree, we get, on bidegree $(1,2)$:
$$\mu_{\mathcal{T}_i,\mathcal{T}_i}=\varepsilon_i \mu.$$
Thus, $\mathcal{T}_i$ is Nijenhuis on $(A\oplus A^*, \mu)$ and Corollary~\ref{cor_T(JN)=0->T(N)=0} yields that $N_i$ is Nijenhuis on $(A, \mu)$.
\end{proof}

Next, we prove that a (para-)hypersymplectic structure with torsion on a Lie algebroid $(A, \mu)$ determines some compatibility properties among the $N_i$'s, the $\pi_i$'s  and the $\omega_i$'s.

\begin{prop}\label{prop_PHST->[Ni,Ni]=0}
    Let $(\omega_1, \omega_2, \omega_3)$ be a (para-)hypersymplectic structure with torsion on $(A, \mu)$. The  Nijenhuis morphisms $N_1,N_2$ and $N_3$ given by (\ref{def_Ni}) are pairwise compatible in the sense that $[N_i,N_j]_{FN}=0, i,j\in\{1,2,3\}$.
\end{prop}
\begin{proof}
    First, notice that $[N_i,N_i]_{FN}=-2\ {\text{\Fontlukas T}}_{\mu}N_i$, thus, when $i=j$, the statement was proved in Theorem \ref{Thm_PHST->Ni_Nijenhuis}. Let us consider now $i,j \in \{1,2,3\}$, with $i\neq j$. From Proposition~\ref{HS_structure_mu+gamma+psi}, the triplet $(\mathcal{S}_1, \mathcal{S}_2, \mathcal{S}_3)$ is a (para-)hypersymplectic structure on the pre-Courant algebroid $(A\oplus A^*, \mu+\psi)$, with $\psi= \frac{\varepsilon_k}{2}[\pi_k, \pi_k]$, for any $k\in \{1,2,3\}$. Using Proposition 5.4 in~\cite{AC14a}, we have $C_{\mu+\psi}(\mathcal{T}_i,\mathcal{T}_j)=0$, with $\mathcal{T}_i$ given by (\ref{def_T_i}), i.e.,
$$\bb{N_i}{\bb{N_j}{\mu+\psi}}+\bb{N_j}{\bb{N_i}{\mu+\psi}}=0.$$
Splitting up this equality in terms of bidegrees, we get, on bidegree $(1,2)$, $$\bb{N_i}{\bb{N_j}{\mu}}+\bb{N_j}{\bb{N_i}{\mu}}=0,$$ which can be written as
$C_{\mu}(\mathcal{T}_i,\mathcal{T}_j)=0$. Applying Proposition \ref{prop_C_mu(Ti,Tj)=-2[I,J]_FN}, the statement is proved.
\end{proof}

\begin{prop}\label{prop_PHST->[pi_i,pi_j]=0}
    Let $(\omega_1, \omega_2, \omega_3)$ be a (para-)hypersymplectic structure with torsion on $(A, \mu)$ and $\pi_k$ be the inverse of $\omega_k$, $k=1,2,3$. Then, $[\pi_i, \pi_j]=0$, $i,j\in\{1,2,3\}, i\neq j$.
\end{prop}
\begin{proof}
We can assume, without loss of generality, that $j=i-1$ and prove $[\pi_i, \pi_{i-1}]=0$, for $i\in \mathbb{Z}_3$. From Proposition~\ref{HS_structure_mu+gamma+psi}, the triplet $(\mathcal{S}_1, \mathcal{S}_2, \mathcal{S}_3)$ is a (para-)hypersymplectic structure on the pre-Courant algebroid $(A\oplus A^*, \mu+\psi)$, with $\psi= \frac{\varepsilon_k}{2}[\pi_k, \pi_k]$, for any $k\in \{1,2,3\}$. Using Proposition 5.4 in~\cite{AC14a}, we have $C_{\mu+\psi}(\mathcal{S}_i,\mathcal{S}_{i-1})=0$, i.e.,
$$\bb{\omega_i + \varepsilon_i \pi_i}{\bb{\omega_{i-1} + \varepsilon_{i-1} \pi_{i-1}}{\mu+\psi}}+\bb{\omega_{i-1} + \varepsilon_{i-1} \pi_{i-1}}{\bb{\omega_i + \varepsilon_i \pi_i}{\mu+\psi}}=0.$$
Splitting up this equality in terms of bidegrees, we get, on bidegree $(3,0)$,
$$\varepsilon_i\varepsilon_{i-1}\bb{\pi_i}{\bb{\pi_{i-1}}{\mu}}+\varepsilon_i\bb{\pi_i}{\bb{\omega_{i-1}}{\psi}} + {\displaystyle \circlearrowleft_{i,i-1}}=0,$$
where ${\displaystyle \circlearrowleft_{i,i-1}}$ stands for the permutation on indices $i$ and $i-1$.
Using the Jacobi identity of the big bracket, we get
\begin{equation}\label{eq_aux_proof_[pi_i,pi_j]=0}
2\varepsilon_i\varepsilon_{i-1}\bb{\pi_i}{\bb{\pi_{i-1}}{\mu}}+\varepsilon_i\bb{\bb{\pi_i}{\omega_{i-1}}}{\psi} + \varepsilon_{i-1}\bb{\bb{\pi_{i-1}}{\omega_i}}{\psi}=0.
\end{equation}
Noticing that the transition morphisms $N_{i+1}$ given by (\ref{def_Ni}) can be equivalently defined as $N_{i+1}=-\bb{\pi_i}{\omega_{i-1}}=-\varepsilon_{i+1}\bb{\pi_{i-1}}{\omega_i}$ (see~\cite{AC13}), (\ref{eq_aux_proof_[pi_i,pi_j]=0}) becomes
\begin{equation}\label{eq_aux_2_proof_[pi_i,pi_j]=0}
2\varepsilon_i\varepsilon_{i-1}\bb{\pi_i}{\bb{\pi_{i-1}}{\mu}} -(\varepsilon_i + \varepsilon_{i-1}\varepsilon_{i+1}) \bb{N_{i+1}}{\psi}=0.
\end{equation}
Because $\varepsilon_1\varepsilon_2\varepsilon_3=-1$, we have $\varepsilon_i + \varepsilon_{i-1}\varepsilon_{i+1}=0$, for all $i \in \mathbb{Z}_3$, so that (\ref{eq_aux_2_proof_[pi_i,pi_j]=0}) simplifies to $\bb{\pi_i}{\bb{\pi_{i-1}}{\mu}}=0$, which is equivalent to $[\pi_i, \pi_{i-1}]=0$.
\end{proof}

\begin{prop}
    Let $(\omega_1, \omega_2, \omega_3)$ be a (para-)hypersymplectic structure with torsion on $(A, \mu)$ and consider the endomorphisms $N_1,N_2, N_3$ given by (\ref{def_Ni}). Then, $C_{\mu}(\omega_i,N_j)=0,  i,j\in\{1,2,3\}, i\neq j$.
\end{prop}

\begin{proof}
    The triplet $(\mathcal{S}_1, \mathcal{S}_2, \mathcal{S}_3)$ is, by Proposition~\ref{HS_structure_mu+gamma+psi}, a (para-)hypersymplectic structure on $(A\oplus A^*, \mu+\psi)$. Proposition 5.4 in~\cite{AC14a} yields that $C_{\mu+\psi}(\mathcal{S}_i,\mathcal{T}_{j})=0$, $i \neq j$.
    Considering the part of bidegree $(0,3)$ in equation $C_{\mu+\psi}(\mathcal{S}_i,\mathcal{T}_{j})=0$, we get
$$\bb{\omega_i}{\bb{N_j}{\mu}}+\bb{N_j}{\bb{\omega_i}{\mu}}=0.$$
\end{proof}

Recall that if $N$ is a Nijenhuis morphism on a Lie algebroid $(A,\mu)$ then $\mu_N=\{N,\mu\}$ is a Lie algebroid structure on $A$~\cite{magriYKS}. The Lie bracket on $\Gamma(A)$ will be denoted by $[\cdot, \cdot]_N$.

\begin{thm}\label{thm_PHST_on_A_iff_PHST_on_A_N}
    A triplet $(\omega_1, \omega_2, \omega_3)$ is a (para-)hypersymplectic structure with torsion on $(A, \mu)$ if and only if $(\omega_1, \omega_2, \omega_3)$ is a (para-)hypersymplectic structure with torsion on $(A, \mu_{N_i}), i=1,2,3$.
\end{thm}
\begin{proof}
    Let $(\omega_1, \omega_2, \omega_3)$ be a (para-)hypersymplectic structure with torsion on $(A, \mu)$ and fix $i \in \{1,2,3\}$. From Theorem 5.5 in~\cite{AC14a} and Proposition~\ref{HS_structure_mu+gamma+psi} we get that $(\mathcal{S}_1, \mathcal{S}_2, \mathcal{S}_3)$ is a (para-)hypersymplectic structure on the pre-Courant algebroid $(A\oplus A^*, (\mu+\psi)_{\mathcal{T}_i})$, with $\psi= \frac{\varepsilon_k}{2}[\pi_k, \pi_k]$, for any $k \in \{1,2,3\}$. In the computation that follows we consider $k=i$. The pre-Courant structure $(\mu+\psi)_{\mathcal{T}_i}$ is given by $\bb{N_i}{\mu+\psi}$ and we  have
    \begin{align*}
        \bb{N_i}{\mu+\psi}&=\mu_{N_i}+ \bb{N_i}{-\frac{\varepsilon_i}{2}\bb{\pi_i}{\bb{\pi_i}{\mu}}}\\
        &=\mu_{N_i} - \frac{\varepsilon_i}{2} \bb{\bb{N_i}{\pi_i}}{\bb{\pi_i}{\mu}} -\frac{\varepsilon_i}{2} \bb{\pi_i}{\bb{\bb{N_i}{\pi_i}}{\mu}}\\
        &\quad\quad - \frac{\varepsilon_i}{2} \bb{\pi_i}{\bb{\pi_i}{\bb{N_i}{\mu}}}\\
        &=\mu_{N_i} + \frac{\varepsilon_i}{2} [\pi_i,\pi_i]_{N_i},
    \end{align*}
    where we used, in the last equality, the fact that $\bb{N_i}{\pi_i}=0$ (see~\cite{AC13}). Applying Proposition~\ref{HS_structure_mu+gamma+psi} we conclude that $(\omega_1, \omega_2, \omega_3)$ is a (para-)hypersymplectic structure with torsion on $(A, \mu_{N_i})$.

    The converse holds because the statements of Theorem 5.5 in~\cite{AC14a} and Proposition~\ref{HS_structure_mu+gamma+psi} are equivalences.
\end{proof}

The proof of Theorem~\ref{thm_PHST_on_A_iff_PHST_on_A_N} can be summarized in the diagram:

$$\xymatrixcolsep{4pc}\xymatrix@R=2.5pc{
    *\txt{ $(\omega_1, \omega_2, \omega_3)$ (para-)HST\\ on $(A,\mu)$} \ar@{<=>}[r]^{\scriptsize Thm~\ref{thm_PHST_on_A_iff_PHST_on_A_N}} \ar@{<=>}[d]_{\scriptsize Prop~\ref{HS_structure_mu+gamma+psi}}&*\txt{ $(\omega_1, \omega_2, \omega_3)$ (para-)HST\\ on $(A,\mu_{N_i})$} \ar@{<=>}[d]^{\scriptsize Prop~\ref{HS_structure_mu+gamma+psi}}\\
    *\txt{ $(\mathcal{S}_1, \mathcal{S}_2, \mathcal{S}_3)$ (para-)HS\\ on $(A\oplus A^*,\mu+\psi)$} \ar@{<=>}[r]^{\scriptsize Thm~5.5}_{\scriptsize in~[5]}&  *\txt{ $(\mathcal{S}_1, \mathcal{S}_2, \mathcal{S}_3)$ (para-)HS\\ on $(A\oplus A^*,(\mu+\psi)_{\mathcal{T}_i})$}
    }$$

\

\noindent where we used the abbreviations HS and HST for hypersymplectic and hypersymplectic with torsion, respectively.

\begin{prop}\label{prop_Twist_Poisson_induces_Lie_algebroid_in_dual}
    Let $(A,\mu)$ be a Lie algebroid and $\pi$ a non-degenerate bivector on $A$ with inverse $\omega$. Then, $\gamma^{\pi}:=\mu_{\pi}+\frac{1}{2}\bb{\omega}{[\pi, \pi]}$ is a Lie algebroid structure on $A^*$.
\end{prop}
\begin{proof}
Aiming to prove that $\bb{\gamma^{\pi}}{\gamma^{\pi}}=0$, we compute
\begin{align}
\bb{\gamma^{\pi}}{\gamma^{\pi}}&=\bb{\mu_{\pi}+\frac{1}{2}\bb{\omega}{[\pi, \pi]}}{\mu_{\pi}+\frac{1}{2}\bb{\omega}{[\pi, \pi]}}\nonumber\\
        &=\bb{\mu_{\pi}}{\mu_{\pi}}+ \bb{\mu_{\pi}}{\bb{\omega}{[\pi, \pi]}} + \frac{1}{4} \bb{\bb{\omega}{[\pi, \pi]}}{\bb{\omega}{[\pi, \pi]}}\nonumber\\
        &=\bb{\mu_{\pi}}{\mu_{\pi}}+ \bb{\mu_{\pi}}{\bb{\omega}{[\pi, \pi]}} + \frac{1}{4} \bb{\bb{\bb{\omega}{[\pi, \pi]}}{\omega}}{[\pi, \pi]}\label{eq_aux_proof_(gamma^pi,gamma^pi)=0}\\
        &\quad\quad\quad\quad + \frac{1}{4} \bb{\omega}{\bb{\bb{\omega}{[\pi, \pi]}}{[\pi, \pi]}},\nonumber
\end{align}
where we used the Jacobi identity in the last equality.
Using again the Jacobi identity and taking into account the bidegree, we get $\bb{\bb{\omega}{[\pi, \pi]}}{[\pi, \pi]}=0$. Furthermore, a straightforward computation leads to $\bb{\bb{\omega}{[\pi, \pi]}}{\omega}=4\bb{\pi}{\bb{\omega}{\mu}}$. Therefore, equality (\ref{eq_aux_proof_(gamma^pi,gamma^pi)=0}) becomes
\begin{equation}\label{eq_aux2_proof_(gamma^pi,gamma^pi)=0}
\bb{\gamma^{\pi}}{\gamma^{\pi}}=\bb{\mu_{\pi}}{\mu_{\pi}}+ \bb{\mu_{\pi}}{\bb{\omega}{[\pi, \pi]}} + \bb{\bb{\pi}{\bb{\omega}{\mu}}}{[\pi, \pi]}.
\end{equation}
A simple computation shows that $\bb{\mu_{\pi}}{\mu_{\pi}}=\bb{\mu}{[\pi, \pi]}$ and the Jacobi identity applied to the last two terms of (\ref{eq_aux2_proof_(gamma^pi,gamma^pi)=0}) yields
\begin{multline*}
    \bb{\gamma^{\pi}}{\gamma^{\pi}}=\bb{\mu}{[\pi, \pi]}+ \bb{\bb{\mu_{\pi}}{\omega}}{[\pi, \pi]} + \bb{\omega}{\bb{\mu_{\pi}}{[\pi, \pi]}}\\
            + \bb{\bb{-{\rm id}_A}{\mu}}{[\pi, \pi]} + \bb{\bb{\omega}{\bb{\pi}{\mu}}}{[\pi, \pi]}.
\end{multline*}
Since $\bb{{\rm id}_A}{\mu}= \mu$ and $\bb{\mu_{\pi}}{\omega}=\bb{\bb{\pi}{\mu}}{\omega}=-\bb{\omega}{\bb{\pi}{\mu}}$, we obtain
$$\bb{\gamma^{\pi}}{\gamma^{\pi}}=\bb{\omega}{\bb{\mu_{\pi}}{[\pi, \pi]}}.$$
Finally, a cumbersome computation shows that $\bb{\mu_{\pi}}{[\pi, \pi]}=0$ and completes the proof.
\end{proof}

\begin{rem}
In the proof of Proposition~\ref{prop_Twist_Poisson_induces_Lie_algebroid_in_dual} we only use the properties of the big bracket. However, the proof can be done using the operation of twisting by a bivector and by a $2$-form that was introduced in \cite{roy}. Let us briefly explain this, using
the notation of \cite{roy}. The twisting of $(\mu,0,0,0)$ by $- \omega$ yields the quasi-Lie bialgebroid structure $(\mu,0,0, \{\mu, \omega\})$ on $(A,A^*)$ and the twisting of $(\mu,0,0,\{\mu, \omega\})$ by $\pi$ gives
\begin{equation}  \label{twist_pi}
\left(\mu+\{\pi,\{\mu,\omega\}\},\, \{\mu,\pi \} + \frac{1}{2}\{\omega, [\pi,\pi]\},\, 0,\, \{\mu,\omega\}\right).
\end{equation}
From Lemma \ref{pi_omega_inverses}, the  twisted Maurer-Cartan equation
$$[\pi, \pi]= 2 d \omega\left(\pi^\sharp(\cdot), \pi^\sharp(\cdot),\pi^\sharp(\cdot)\right)$$ holds, so that (\ref{twist_pi}) is a quasi-Lie bialgebroid structure on  $(A,A^*)$, as it is proved in~\cite{roy}. This, in turn, implies that the term of bidegree $(2,1)$ in (\ref{twist_pi}),
\begin{equation}  \label{quasi_Lie_A*}
 \{\mu,\pi \} + \frac{1}{2}\{\omega, [\pi,\pi]\},
\end{equation}
is a Lie algebroid structure on $A^*$.
\end{rem}

In the next theorem we show that there is a one-to-one correspondence between (para-)hypersymplectic structures with torsion on $A$ and on $A^*$.

\begin{thm}\label{Thm_PHST_on_A_iff_PHST_on_A*}
    The triplet $(\omega_1, \omega_2, \omega_3)$ is a (para-)hypersymplectic structure with torsion on $(A, \mu)$ if and only if $(\pi_1, \pi_2, \pi_3)$ is a (para-)hypersymplectic structure with torsion on $\left(A^*, \varepsilon_i \gamma^{\pi_i}\right)$, where $\gamma^{\pi_i}:=\mu_{\pi_i}+\frac{1}{2}\bb{\omega_i}{[\pi_i, \pi_i]}, i=1,2,3$, and $\pi_i$ is the inverse of $\omega_i$.
\end{thm}
Under the conditions of Theorem \ref{Thm_PHST_on_A_iff_PHST_on_A*}, the Lie algebroid structures on $A^*$ can be written as $\varepsilon_i\mu_{\pi_i}+\bb{\omega_i}{\psi}, i=1,2,3$, with $\psi=\frac{\varepsilon_k}{2}[\pi_k, \pi_k]$, for any $k \in \{1,2,3\}$.

\begin{proof}[Proof of Theorem \ref{Thm_PHST_on_A_iff_PHST_on_A*}]
Let $(\omega_1, \omega_2, \omega_3)$ be a (para-)hypersymplectic structure with torsion on $(A, \mu)$ and fix $i \in \{1,2,3\}$. From Theorem 5.5 in~\cite{AC14a} and Proposition~\ref{HS_structure_mu+gamma+psi} we have that $(\mathcal{S}_1, \mathcal{S}_2, \mathcal{S}_3)$ is a (para-)hypersymplectic structure on the pre-Courant algebroid $(A\oplus A^*, (\mu+\psi)_{\mathcal{S}_i})$, with $\psi= \frac{\varepsilon_k}{2}[\pi_k, \pi_k]$, for any $k \in \{1,2,3\}$. In the computation that follows we consider $k=i$. The pre-Courant structure $(\mu+\psi)_{\mathcal{S}_i}$ is given by 
    \begin{align*}
        (\mu+\psi)_{\mathcal{S}_i}
        &=\varepsilon_i \gamma^{\pi_i}+\mu_{\omega_i}.
    \end{align*}
   A direct computation shows that $\mu_{\omega_i}=\frac{\varepsilon_i}{2}[\omega_i, \omega_i]_{\gamma^{\pi_i}}$, where $[\cdot, \cdot]_{\gamma^{\pi_i}}$ stands for the Schouten-Nijenhuis bracket of the Lie algebroid structure $\gamma^{\pi_i}$ on $A^*$. Then,
     $$(\mu+\psi)_{\mathcal{S}_i}=\varepsilon_i \gamma^{\pi_i} + \frac{\varepsilon_i}{2}[\omega_i, \omega_i]_{\gamma^{\pi_i}}=\varepsilon_i \gamma^{\pi_i} + \frac{1}{2}[\omega_i, \omega_i]_{\varepsilon_i \gamma^{\pi_i}}$$
    and, using Proposition~\ref{HS_structure_mu+gamma+psi}, the result follows.

    The converse holds because the statements of Theorem 5.5 in~\cite{AC14a} and Proposition~\ref{HS_structure_mu+gamma+psi} are equivalences.
\end{proof}

The proof of Theorem~\ref{Thm_PHST_on_A_iff_PHST_on_A*} can be summarized in the diagram:

$$\xymatrixcolsep{4pc}\xymatrix@R=2.5pc{
    *\txt{ $(\omega_1, \omega_2, \omega_3)$ (para-)HST\\ on $(A,\mu)$} \ar@{<=>}[r]^{\scriptsize Thm~\ref{Thm_PHST_on_A_iff_PHST_on_A*}} \ar@{<=>}[d]_{\scriptsize Prop~\ref{HS_structure_mu+gamma+psi}}&*\txt{ $(\pi_1, \pi_2, \pi_3)$ (para-)HST\\ on $\left(A^*, \varepsilon_i \gamma^{\pi_i}\right)$} \ar@{<=>}[d]^{\scriptsize Prop~\ref{HS_structure_mu+gamma+psi}}\\
    *\txt{ $(\mathcal{S}_1, \mathcal{S}_2, \mathcal{S}_3)$ (para-)HS\\ on $(A\oplus A^*,\mu+\psi)$} \ar@{<=>}[r]^{\scriptsize Thm~5.5}_{\scriptsize in~[5]}&  *\txt{ $(\mathcal{S}_1, \mathcal{S}_2, \mathcal{S}_3)$ (para-)HS\\ on $(A\oplus A^*,(\mu+\psi)_{\mathcal{S}_i})$}
    }$$



\

\

\

\noindent {\bf Acknowledgments.}
This work was partially supported by the Centre for Mathematics of the University of Coimbra -- UID/MAT/00324/2013, funded by the Portuguese Government through FCT/MEC and co-funded by the European Regional Development Fund through the Partnership Agreement PT2020.



\begin{thebibliography}{99}


%
\bibitem{ALC11} P. Antunes, C. Laurent-Gengoux, J. M. Nunes da Costa, Hierarchies and compatibility on Courant algebroids, Pac. J. Math., 261 (2013), no. 1, 1--32.

\bibitem{AC13} P. Antunes, J. M. Nunes da Costa, Hyperstructures on Lie algebroids, Rev. in Math. Phys. 25 (2013), no. 10, 1343003 (19 pages).

\bibitem{AC14} P. Antunes, J. M. Nunes da Costa, Induced hypersymplectic and hyperk\"{a}hler structures on the dual of a Lie algebroid, Int. J. Geom. Meth. Mod. Phys., {\bf 11} (2014), no. 9, 1460030 (9 pages).

\bibitem{AC14a} P. Antunes, J. M. Nunes da Costa, Hypersymplectic structures on Courant algebroids, J. Geom. Mech. 7 (2015), no. 3, 255--280.

%

\bibitem{grab} J. Grabowski, Courant-Nijenhuis tensors and generalized geometries,
in Groups, geometry and physics, Monogr. Real Acad. Ci. Exact. F\'{\i}s.-Qu\'{\i}m. Nat. Zaragoza, 29, 2006, pp. 101--112.

\bibitem{GP00} G. Grantcharov, Y.S. Poon, Geometry of hyper-K\"{a}hler connection with torsion, Comm. Math. Phys. 213 (2000), 19--37.

\bibitem{HP96} P.S. Howe, G. Papadopoulos, Twistor spaces for hyper-K\"{a}hler manifolds with torsion, Phys. Lett., B 379 (1996), 80--86.

%
\bibitem{magriYKS} Y. Kosmann-Schwarzbach, F. Magri, Poisson-Nijenhuis structures,
Ann. Inst. H. Poincar\'e, Phys. Th\'eor. 53 (1990), no. 1, 35--81.

\bibitem{YKS92}  Y. Kosmann-Schwarzbach, Jacobian quasi-bialgebras and quasi-Poisson Lie groups. In {\em Mathematical aspects of classical field theory}, Contemp. Math., 132,  Amer. Math. Soc., 1992, pp. 459--489.

%
%
%

\bibitem{YKS11} Y. Kosmann-Schwarzbach, Nijenhuis structures on Courant algebroids,  Bull. Brazilian Math. Soc., 42 (2011), no. 4, 625-649.






\bibitem{madsen_thesis} T.B. Madsen, Torsion geometry and scalar functions, Ph.D. thesis, University of Southern Denmark, 2011.


\bibitem{CS08} F. Mart\'{\i}n Cabrera, A. Swann, The intrinsic torsion of almost quaternion-Hermitian manifolds, Ann. Inst. Fourier, Grenoble 58 (2008), 1455--1497.


%

\bibitem{royContemp} D. Roytenberg, On the structure of graded symplectic supermanifolds and Courant algebroids, in {\em Quantization, Poisson brackets and beyond}, T. Voronov, ed., Contemp. Math., 315, Amer. Math. Soc., Providence, RI, 2002, pp. 169--185.

\bibitem{roy} D. Roytenberg, Quasi-Lie bialgebroids and twisted Poisson manifolds, Lett. Math. Phys. 61 (2002), no. 2, 123--137.

\bibitem{sw01} P. \v{S}evera, A. Weinstein, Poisson geometry with a $3$-form background, in {\em Noncommutative geometry and string theory} Prog. Theor. Phys.,
Suppl. 144 (2001), 145--154.

%
\bibitem{vaintrob} A. Vaintrob, Lie algebroids and homological vector fields, Russian Math. Surveys 52 (1997), no. 2, 428--429.


\bibitem{voronov} T. Voronov, Graded manifolds and Drinfeld doubles for Lie bialgebroids, in {\em Quantization, Poisson brackets and beyond}, T. Voronov, ed., Contemp. Math. 315, Amer. Math. Soc., Providence, RI, 2002, pp. 131--168.

\end{thebibliography}
\end{document}